\documentclass{amsart}
\usepackage[english]{babel}

\usepackage[latin9]{inputenc}

\usepackage{amstext}
\usepackage{amssymb}
\usepackage{amsmath}
\usepackage[nobysame]{amsrefs}
\usepackage{mathrsfs}
\usepackage{graphicx}
\usepackage{array}
\usepackage{mathrsfs}
\usepackage{verbatim}
\usepackage[margin=1.1in]{geometry}

\numberwithin{equation}{section}

\theoremstyle{plain}
\newtheorem{theorem}{Theorem}[section]

\newtheorem{lemma}{Lemma}[section]

\newtheorem{prop}[lemma]{Proposition}

\theoremstyle{definition}
\newtheorem{defn}[lemma]{Definition}

\theoremstyle{remark}
\newtheorem*{remark}{Remark}

\DeclareMathOperator{\tr}{tr}

\DeclareMathOperator{\spec}{spec}

\DeclareMathOperator{\Id}{Id}

\newcommand{\ud}{\,\mathrm{d}}
\newcommand{\RR}{\mathbb{R}}
\newcommand{\CC}{\mathbb{C}}
\newcommand{\NN}{\mathbb{N}}
\newcommand{\ZZ}{\mathbb{Z}}

\newcommand{\Or}{\mathcal{O}}
\newcommand{\bd}[1]{\boldsymbol{#1}}
\newcommand{\I}{\mathrm{i}}
\newcommand{\abs}[1]{\left\lvert#1\right\rvert}
\newcommand{\norm}[1]{\left\lVert#1\right\rVert}

\newcommand{\mc}[1]{\mathcal{#1}}
\newcommand{\average}[1]{\langle#1\rangle}

\newcommand{\braket}[1]{\left\langle#1\right\rangle}

\usepackage{color}
\newcommand{\wh}[1]{\widehat{#1}}
\newcommand{\wt}[1]{\widetilde{#1}}

\IfFileExists{mathabx.sty}%
  {\DeclareFontFamily{U}{mathx}{\hyphenchar\font45}%
   \DeclareFontShape{U}{mathx}{m}{n}{<->mathx10}{}%
   \DeclareSymbolFont{mathx}{U}{mathx}{m}{n}%
   \DeclareFontSubstitution{U}{mathx}{m}{n}%
   \DeclareMathAccent{\widebar}{0}{mathx}{"73}%
}{%
  \PackageWarning{mathabx}{%
    Package mathabx not available, therefore\MessageBreak substituting
    widebar with overline\MessageBreak }%
  \newcommand{\widebar}[1]{\wb{#1}}%
}
\newcommand{\wb}[1]{\widebar{#1}}

\newcommand{\veps}{\varepsilon}

\newcommand{\FGA}{\mathrm{FGA}}

\begin{document}

\title[FGA for high frequency wave propagation in periodic media]{Frozen Gaussian approximation
  for high frequency wave propagation in periodic media}

\date{\today}

\author{Ricardo Delgadillo}
\address{Department of Mathematics, University of California, Santa Barbara,
CA 93106}
\email{rdelgadi@math.ucsb.edu}

\author{Jianfeng Lu}
\address{Departments of Mathematics, Physics, and
  Chemisty, Duke University, Box 90320, Durham, NC 27708}
\email{jianfeng@math.duke.edu}

\author{Xu Yang}
\address{Department of Mathematics, University of California, Santa Barbara,
CA 93106}
\email{xuyang@math.ucsb.edu}

\thanks{R.D. and X.Y. were partially supported by the NSF grants DMS-1418936 and DMS-1107291: NSF Research Network in Mathematical Sciences ``Kinetic description of emerging challenges in multiscale problems of natural science''.
The work of J.L.~was supported in part by the Alfred P.~Sloan foundation and the National Science
Foundation under award DMS-1312659. X.Y. was also partially supported
by the Regents Junior Faculty Fellowship and Hellman Family Foundation Faculty Fellowship of University of California, Santa Barbara.}

\begin{abstract}
  Propagation of high-frequency wave in periodic media is a
  challenging problem due to the existence of multiscale characterized
  by short wavelength, small lattice constant and large physical
  domain size. Conventional computational methods lead to extremely
  expensive costs, especially in high dimensions. In this paper, based
  on Bloch decomposition and asymptotic analysis in the phase space,
  we derive the frozen Gaussian approximation for high-frequency wave
  propagation in periodic media and establish its converge to the true
  solution. The formulation leads to efficient numerical algorithms,
  which are presented in a companion paper \cite{DeLuYa:14}.
\end{abstract}

\maketitle
\section{Introduction}

We are interested in studying high-frequency wave propagation in periodic media. A typical example is given by the following Schr\"{o}dinger
equation in the semiclassical regime with a superposition of a (highly
oscillatory) microscopic periodic potential and a macroscopic smooth
potential,
\begin{equation}\label{eq:schrodinger}
  \I\veps\dfrac{\partial\psi^{\veps}}{\partial t}=-\dfrac{\veps^{2}}{2}\Delta\psi^{\veps}+V(\bd{x}/\veps)\psi^{\veps}+U(\bd{x})\psi^{\veps},\hspace{.5cm}\bd{x}\in\mathbb{R}^{d},
\end{equation}
where $V$ and $U$ are smooth potential functions, $V$ is periodic with
respect to the lattice $\mathbb{Z}^{d}$:
$V(\bd{x}+\mathbf{e}_{i})=V(\bd{x})$ for any $\bd{x}\in\mathbb{R}^{d}$
and $\{\bd{e}_{i},\, i=1,2,\cdots, d\}$ is the standard basis of
$\mathbb{R}^{d}$. Here $\veps\ll1$ is the rescaled Planck constant,
$\psi^{\veps}$ is the wave function, and $d$ is the spatial dimensionality.

The equation \eqref{eq:schrodinger} can be viewed as a model for
electron dynamics in a crystal, where $V$ is the effective periodic
potential induced by the crystal, and $U$ is some external macroscopic
potential. Notice
that we have identified the period of $V(\bd{x}/\veps)$ and the
``semiclassical parameter'' in front of the derivative terms. This
parameter choice gives the most interesting case as $\veps \to 0$ \cite{BeMaPo:01}.

The mathematical analysis of this work is motivated by the challenge
of numerical simulation of \eqref{eq:schrodinger} when $\veps$ is
small. In this semiclassical regime, the wave function $\psi^\veps$
becomes oscillatory with wave length $\Or(\veps)$. This means a
computational domain of order $1$ size contains $\mathcal{O}(1/\veps)$
wavelengths, and each of them needs to be resolved if conventional
numerical methods are applied. For example, even for the simplest case
$V=0$ (no lattice potential), a mesh size of $\Or(\veps)$ is required
when using the time-splitting spectral method \cite{BaJiMa:02} to
compute (\ref{eq:schrodinger}) directly; an even worse mesh size of
$o(\veps)$ is needed if one uses the Crank-Nicolson schemes
\cite{MaPiPo:99} or the Dufort-Frankel scheme
\cite{MaPiPo:00}. Besides, the presence of non-zero lattice potential
introduces further difficulties which restrict the mesh size to be
$o(\veps)$ in the standard time-splitting spectral method
\cite{BaJiMa:02}. Special techniques using Bloch decomposition are
needed to relax the mesh size to be of $\Or(\veps)$
\cite{HuJiMaSp:07,HuJiMaSp:08,HuJiMaSp:09}. Moreover, in these
methods, a large domain is demanded in order to avoid the boundary
effects.  Therefore the total number of grid points is huge, which
usually leads to unaffordable computational cost, especially in high
($d>1$) dimensions.

An alternative efficient approach is to solve \eqref{eq:schrodinger}
asymptotically by the Bloch decomposition and modified WKB methods
\cite{BeLiPa:78,CarlesSparber:12,ELuYa:13}, which lead to eikonal and 
transport equations in the semi-classical regime. An advantage of this
method is that the computational cost is independent of $\veps$.
However, the eikonal equation can develop singularities which make the
method break down at caustics. The Gaussian beam method (GBM)
~\cite{Po:82} was then introduced by Popov to overcome this drawback
at caustics. The idea is to allow the phase function to be {complex}
and choose the imaginary part properly so that the solution has a
Gaussian profile; see
\cite{TaQiRa:07,Ta:08,MoRu:10,JiWuYa:08,JiWuYa:10,JiWuYaHu:10,JiWuYa:11,JiMaSp:11,JeJi:14}
for recent developments.  Similar ideas can be also found in the
Hagedorn wave packet method \cite{Ha:80,FaGrLu:09}.  Unlike the
geometric optics based method, the Gaussian beam method allows for
accurate computation of wave function around caustics
~\cite{Ra:82,DiGuRa:06}. But the problem is that the constructed beam
must stay near the geometric rays to maintain accuracy. This becomes a
drawback when the solution spreads~\cite{LuYa:11, MoRu:10, QiYi2:10}.

The Herman-Kluk propagator \cite{HeKl:84, Ka:94, Ka:06} was proposed
for Schr\"odinger equation without the oscillatory periodic background
potential. The method was rigorously analyzed in \cite{SwRo:09, Ro:10}
and further extended as the frozen Gaussian approximation (FGA) for
general high frequency wave propagation
in~\cite{LuYa:11,LuYa:12,LuYa2:12}.  The FGA method uses Gaussian
functions with fixed widths, instead of using those that might spread
over time, to approximate the wave solution.  Despite its superficial
similarity with the Gaussian beam method, it is different at a
fundamental level.  FGA is based on phase plane analysis, while GBM is
based on the asymptotic solution to a wave equation with Gaussian
initial data.  In FGA, the solution to the wave equation is
approximated by a superposition of Gaussian functions living in phase
space, and each function is {not} necessarily an asymptotic solution,
while GBM uses Gaussian functions (called beams) in physical space,
with each individual beam being an asymptotic solution to the wave
equation. The main advantage of FGA over GBM is that the problem of
beam spreading no longer exists.

In this paper, we extend FGA for computation of high-frequency wave
propagation in periodic media. We mainly focus on the derivation of an
integral representation formula of FGA in the phase space and
establish the rigorous convergence results for FGA. While the FGA
works for general strictly hyperbolic equations, we focus in this
paper the case of semiclassical Schr\"odinger equation with periodic
media \eqref{eq:schrodinger}.  The computational algorithm and
numerical results will be presented in a separate paper
\cite{DeLuYa:14}. The rest of the paper is organized as follows.  We
first recall the Bloch decomposition of periodic media and introduce
the windowed Bloch transform in Section~\ref{sec:WBT}.  In
Section~\ref{sec:main_results}, we present the formulation of FGA for
periodic media and the main convergence result. The proof of the main
result is given in Section~\ref{sec:proof}.

\textbf{Notations.}  The absolute value, Euclidean distance, vector
norm, induced matrix norm, and sum of components of a multi-index will
all be denoted by $|\cdot|$. We will use the standard notations
$\mathcal{S}$, $\mathcal{C}^{\infty}$, and $\mathcal{C}_{c}^{\infty}$
for Schwartz class functions, smooth functions, and compactly
supported smooth functions, respectively. We will sometimes use
subscripts to specify the dependence of a constant on the parameters,
for instance, notations like $C_{T}$ to specify the dependence of a
constant on a parameter $T$.

\section{Bloch decomposition and windowed Bloch transform}\label{sec:WBT}
The frozen Gaussian approximation for periodic media relies crucially
on the Bloch decomposition to capture the fine scale ($\Or(\veps)$
spatial scale) oscillation.  First we briefly recall
the well-known Bloch-Floquet decomposition for Schr\"odinger operators
with a periodic potential.

Consider a Schr\"odinger operator
\begin{equation}
  H = - \dfrac{1}{2}\Delta + V(\bd{x}),
\end{equation}
where the potential $V$ is periodic with respect to the lattice
$\ZZ^d$. We denote $\Gamma$ the unit cell of the lattice: $\Gamma =
[0, 1)^d$. The unit cell of the reciprocal lattice (known as the first
Brillouin zone) is given by $\Gamma^{\ast} = [- \pi, \pi)^d$. It is
standard (e.g.,~\cite{ReedSimon4}) that the spectrum of $H$ is
given by energy bands
\begin{equation*}
  \spec(H) = \bigcup_{n=1}^{\infty} \bigcup_{\bd{\xi} \in \Gamma^{\ast}} E_n(\bd{\xi}),
\end{equation*}
where for each $\bd{\xi} \in \Gamma^{\ast}$, $\{ E_n(\bd{\xi}) \}$ are
the collection of eigenvalues (in ascending order) of the operator
\begin{equation*}
  H_{\bd{\xi}} = \frac{1}{2}(-\I \nabla_{\bd{x}}  + \bd{\xi})^2+V(\bd{x})
\end{equation*}
with periodic boundary condition on $\Gamma$. The Bloch waves are the
associated eigenfunctions: For each band $n$ and $\bd{\xi}\in
\Gamma^{\ast}$, it solves
\begin{equation}\label{eq:Bloch}
  H_{\bd{\xi}} u_n(\bd{\xi},\cdot) = E_n(\bd{\xi}) u_n(\bd{\xi},\cdot).
\end{equation}
with periodic boundary condition on $\Gamma$, where $\bd{\xi}$ serves
as a parameter in the above equation. $u_n(\bd{\xi}, \cdot)$ is
normalized that
\begin{equation}\label{eq:normalization}
  \int_{\Gamma} \abs{u_n(\bd{\xi}, \bd{x})}^2 \ud x = 1.
\end{equation}
We  extend $u_n(\bd{\xi}, \bd{x})$ periodically with respect to
the second variable, so it is defined on $\Gamma^{\ast} \times
\RR^d$. We will also write $u_{n,\bd{\xi}} = u_n(\bd{\xi}, \cdot)$ when
the former is more convenient.

These Bloch waves generalize the Fourier modes (complex exponentials)
to periodic media (see for example discussions in
\cite{ELuYa:13}). In particular, for any function $f \in
L^2(\RR^d)$, we have the Bloch decomposition
\begin{equation}\label{eq:blochdecomp}
  f(\bd{x}) = \frac{1}{(2\pi)^{d/2}} \sum_{n=1}^{\infty} \int_{\Gamma^{\ast}}
  u_n(\bd{\xi}, \bd{x}) e^{\I \bd{\xi} \cdot \bd{x}} (\mc{B} f)_n(\bd{\xi}) \ud \bd{\xi}.
\end{equation}
where the Bloch transform $\mc{B}: L^2(\RR^d) \to
L^2(\Gamma^{\ast})^{\mathbb{N}}$ is given by
\begin{equation}\label{eq:blochcoef}
  (\mc{B}f)_{n}(\bd{\xi}) = \frac{1}{(2\pi)^{d/2}}\int_{\RR^d} \wb{u}_n(\bd{\xi}, \bd{y}) e^{-\I \bd{\xi} \cdot \bd{y}} f(\bd{y}) \ud \bd{y}.
\end{equation}
As an analog of the Parseval's identity, we have
\begin{equation}\label{eq:parseval}
  \int_{\RR^d} \abs{f(\bd{x})}^2 \ud \bd{x} =
  \sum_{n=1}^{\infty} \int_{\Gamma^{\ast}} \abs{(\mc{B} f)_n(\bd{\xi})}^2 \ud \bd{\xi}.
\end{equation}
As suggested by \eqref{eq:blochdecomp} and \eqref{eq:blochcoef}, we
introduce the notation $\Omega$ to denote the phase space
corresponding to one band ($\Gamma^{\ast}$ is viewed as a torus,
\textit{i.e.}, periodic boundary condition is assumed on
$\Gamma^{\ast}$)
\begin{equation}
  \Omega: = \RR^d \times \Gamma^{\ast} = \bigl\{ (\bd{x}, \bd{\xi}) \mid \bd{x} \in \RR^d, \, \bd{\xi} \in \Gamma^{\ast}\}.
\end{equation}
Correspondingly, we will use the notation $(\bd{q}, \bd{p})$ for a point in $\Omega$.

For later usage, we define the Berry phase $\mc{A}_n$ for the Bloch
waves,
\begin{equation}\label{eq:Berry}
  \mc{A}_n(\bd{\xi}) = \bigl\langle
  u_n(\bd{\xi}, \cdot),
 \I \nabla_{\bd{\xi}} u_n(\bd{\xi}, \cdot) \bigr\rangle_{L^2(\Gamma)}.
\end{equation}
The normalization condition~\eqref{eq:normalization} implies $\mc{A}_n(\bd{\xi})$ is always a
real number. We should be cautious about one subtlety though as the
eigenvalue equation \eqref{eq:Bloch} and the normalization only define
$u_n(\bd{\xi}, \cdot)$ up to a unit complex number, in particular, for
any function $\varphi$ periodic in $\Gamma^{\ast}$,
\begin{equation}
  v_n(\bd{\xi}, \bd{x}) = e^{\I \varphi(\bd{\xi})} u_n(\bd{\xi}, \bd{x}), \qquad (\bd{x}, \bd{\xi}) \in \Omega,
\end{equation}
also provides a set of Bloch waves. This is known as the gauge choice
for the Bloch waves. However, different gauge choice gives different
values of $\mc{A}_n(\bd{\xi})$ and even causes trouble if $\varphi$ is
discontinuous. While for the analysis, it suffices to assume smooth
dependence of $u_n$ on $\bd{\xi}$ (which is possible as the $n$-th
band is separated from the rest of the spectrum), this gauge freedom
makes numerical computation nontrivial. We will further address this
by designing a gauge-invariant algorithm in a companion paper
\cite{DeLuYa:14} on the numerical algorithms.

Differentiating \eqref{eq:Bloch} with respect to $\bd{\xi}$ produces
\begin{equation}\label{eq:diffBloch}
  H_{\bd{\xi}} \nabla_{\bd{\xi}}
  u_n(\bd{\xi}, \bd{x}) + (-\I \nabla_{\bd{x}} + \bd{\xi}) u_n(\bd{\xi}, \bd{x})
  = E_n(\bd{\xi}) \nabla_{\bd{\xi}}
  u_n(\bd{\xi}, \bd{x}) + \nabla_{\bd{\xi}} E_n(\bd{\xi}) u_n(\bd{\xi}, \bd{x}).
\end{equation}
Taking inner product with $u_n(\bd{\xi}, \cdot)$ yields
\begin{equation}\label{eq:diffE}
  \nabla_{\bd{\xi}} E_n(\bd{\xi}) = -\I \average{u_n(\bd{\xi}, \cdot),
    \nabla_{\bd{x}} u_n(\bd{\xi}, \cdot)} +  \bd{\xi}.
\end{equation}
Differentiate \eqref{eq:diffBloch} with respect to $\bd{\xi}$ again gives
\begin{multline}
  H_{\bd{\xi}} \nabla_{\bd{\xi}}^2 u_n(\bd{\xi}, \bd{x}) + 2 (-\I
  \nabla_{\bd{x}} + \bd{\xi}) \nabla_{\bd{\xi}}
  u_n(\bd{\xi}, \bd{x}) + u_n(\bd{\xi}, \bd{x}) I \\
  = E_n(\bd{\xi}) \nabla_{\bd{\xi}}^2 u_n(\bd{\xi}, \bd{x}) + 2
  \nabla_{\bd{\xi}} E_n(\bd{\xi}) \nabla_{\bd{\xi}} u_n(\bd{\xi},
  \bd{x}) + E_n(\bd{\xi}) \nabla_{\bd{\xi}}^2 u_n(\bd{\xi}, \bd{x}).
\end{multline}
Taking inner product with $u_n(\bd{\xi}, \cdot)$, one gets
\begin{multline}\label{eq:diffE2}
  \average{u_n(\bd{\xi}, \cdot), -\I \nabla_{\bd{x}}\nabla_{\bd{\xi}}
    u_n(\bd{\xi}, \bd{x})} + \bd{\xi} \average{u_n(\bd{\xi}, \cdot),
    \nabla_{\bd{\xi}} u_n(\bd{\xi}, \cdot)} + I/2  \\
  = \nabla_{\bd{\xi}} E_n(\bd{\xi}) \average{u_n(\bd{\xi}, \cdot),
    \nabla_{\bd{\xi}} u_n(\bd{\xi}, \bd{x})} + \tfrac{1}{2}
  E_n(\bd{\xi}) \average{u_n(\bd{\xi}, \cdot), \nabla_{\bd{\xi}}^2
    u_n(\bd{\xi}, \cdot)}.
\end{multline}
These identities \eqref{eq:diffE} and \eqref{eq:diffE2} will be useful
later.

\medskip

We shall now introduce the windowed Bloch transform. This is an analog
of the windowed Fourier transform (also known as the short time
Fourier transform) widely used in time-frequency signal analysis.
\begin{defn}\label{defwin}
  The windowed Bloch transform $\mc{W}: L^2(\RR^d) \to
  L^2(\Omega)^{\NN}$ is defined as
  \begin{equation}
    (\mc{W} f)_{n}(\bd{q},\bd{p}) =  \frac{2^{d/4}}{(2\pi)^{3d/4}} \braket{u_n(\bd{p}, \cdot) G_{\bd{q}, \bd{p}}, f} = \frac{2^{d/4}}{(2\pi)^{3d/4}} \int_{\RR^d} \wb{u}_n(\bd{p}, \bd{x})  \wb{G}_{\bd{q},\bd{p}}(\bd{x}) f(\bd{x}) \ud \bd{x},
  \end{equation}
  where $G_{\bd{q}, \bd{p}}$ is a Gaussian centered at $(\bd{q},
  \bd{p}) \in \Omega$, given by
  \begin{equation}
    G_{\bd{q},\bd{p}}(\bd{x}) =
    \exp\Bigl(- \frac{1}{2} \abs{\bd{x}-\bd{q}}^{2}
    +\I\bd{p}\cdot(\bd{x}-\bd{q})\Bigr).
  \end{equation}
  The adjoint operator $\mc{W}^{\ast}: L^2(\Omega)^{\NN} \to
  L^2(\RR^d)$ is then
  \begin{equation}
    (\mc{W}^{\ast} g)(\bd{x}) = \frac{2^{d/4}}{(2\pi)^{3d/4}} \sum_{n=1}^{\infty}
    \iint_{\Omega} u_n(\bd{p}, \bd{x}) G_{\bd{q},\bd{p}}(\bd{x})
    g_{n}(\bd{q},\bd{p})  \ud\bd{q}\ud\bd{p}.
  \end{equation}
\end{defn}

\begin{prop} The windowed Bloch transform and its adjoint satisfies
  \begin{equation}\label{eq:reconstruction}
    \mc{W}^{\ast} \mc{W} = \mathrm{Id}_{L^2(\RR^d)}.
  \end{equation}
\end{prop}

\begin{remark}
  Similar to the windowed Fourier transform, the representation given
  by the windowed Bloch transform is redundant, so that $\mc{W}
  \mc{W}^{\ast} \not= \mathrm{Id}_{L^2(\Omega)^{\NN}}$. The
  normalization constant in the definition of $\mc{W}$ is also due to
  this redundancy.
\end{remark}
\begin{proof}
  Fix a $f \in L^2(\RR^d)$, by definition, we have
  \begin{equation*}
    \begin{aligned}
      (\mc{W}^{\ast} \mc{W} f)(\bd{x}) & =
      \frac{2^{d/2}}{(2\pi)^{3d/2}} \sum_{n=1}^{\infty} \iint_{\Omega}
      u_n\bigl(\bd{p}, \bd{x} \bigr) G_{\bd{q},
        \bd{p}}(\bd{x}) \average{ G_{\bd{q}, \bd{p}} u_n(\bd{p}, \cdot), f}  \ud \bd{q} \ud \bd{p} \\
      & = \frac{2^{d/2}}{(2\pi)^{3d/2}} \sum_{n=1}^{\infty}
      \iint_{\Omega} \int_{\RR^d} u_n\bigl(\bd{p}, \bd{x}\bigr)
      G_{\bd{q}, \bd{p}}(\bd{x}) \wb{G}_{\bd{q}, \bd{p}}(\bd{y})
      \wb{u}_n(\bd{p}, \bd{y}) f(\bd{y}) \ud \bd{y} \ud \bd{q}
      \ud \bd{p}.
    \end{aligned}
  \end{equation*}
  Let us integrate in $\bd{q}$ first.
  \begin{equation*}
    \begin{aligned}
      \int_{\RR^d} G_{\bd{q}, \bd{p}}(\bd{x}) \wb{G}_{\bd{q},
        \bd{p}}(\bd{y}) \ud \bd{q} & = e^{ \I \bd{p} \cdot ( \bd{x} -
        \bd{y})} \int_{\RR^d} e^{-\abs{\bd{x} - \bd{q}}^2/2
        - \abs{\bd{y} - \bd{q}}^2/2} \ud \bd{q} \\
      & = e^{ \I \bd{p} \cdot ( \bd{x} - \bd{y})}
      e^{-\abs{\bd{x}-\bd{y}}^2 / 4} \int_{\RR^d} \exp\biggl(-
      \Bigl\lvert \bd{q} - \frac{\bd{x} + \bd{y}}{2}
      \Bigr\rvert^2\biggr) \ud \bd{q} \\
      & = \pi^{d/2} e^{ \I \bd{p} \cdot ( \bd{x} -
        \bd{y})} e^{-\abs{\bd{x}-\bd{y}}^2 / 4}.
    \end{aligned}
  \end{equation*}
  Hence, denoting $\wt{f}_{\bd{x}}(\bd{y}) =
  e^{-\abs{\bd{x}-\bd{y}}^2 / 4 } f(\bd{y})$, we
  have
  \begin{equation*}
    \begin{aligned}
      (\mc{W}^{\ast} \mc{W} f)(\bd{x}) & = \frac{1}{(2\pi)^d}
      \sum_{n=1}^{\infty} \int_{\Gamma^{\ast}} \int_{\RR^d}
      u_n\bigl(\bd{p}, \bd{x}\bigr) e^{ \I \bd{p} \cdot (
        \bd{x} - \bd{y})} e^{-\abs{\bd{x}-\bd{y}}^2 / 4}
      \wb{u}_n(\bd{p}, \bd{y}) f(\bd{y}) \ud
      \bd{y} \ud \bd{p} \\
      & = \frac{1}{(2\pi)^d} \sum_{n=1}^{\infty}
      \int_{\Gamma^{\ast}} \int_{\RR^d} u_n\bigl(\bd{p},
      {\bd{x}}\bigr) e^{ \I \bd{p} \cdot ( \bd{x} -
        \bd{y})} \wb{u}_n(\bd{p}, \bd{y} )
      \wt{f}_{\bd{x}}(\bd{y}) \ud \bd{y} \ud \bd{p} \\
      & \stackrel{\eqref{eq:blochdecomp}}{=} \wt{f}_{\bd{x}}(\bd{x})
      = e^{-\abs{\bd{x} - \bd{x}}^2/4} f(\bd{x})
      = f(\bd{x}).
    \end{aligned}
  \end{equation*}
\end{proof}

The previous proposition motivates us to consider the contribution of
each band to the reconstruction formulae
\eqref{eq:reconstruction}. This gives to the operator $\Pi_n^{\mc{W}}:
L^2(\RR^d) \to L^2(\RR^d)$ for each $n \in \NN$
\begin{equation}
  (\Pi_n^{\mc{W}} f)(\bd{x}) = \frac{2^{d/4}}{(2\pi)^{3d/4}}
  \iint_{\Omega} u_n(\bd{p}, \bd{x}) G_{\bd{q},\bd{p}}(\bd{x})
  (\mc{W} f)_{n}(\bd{q},\bd{p})  \ud\bd{q}\ud\bd{p}.
\end{equation}
It follows from \eqref{eq:reconstruction} that $\sum_n \Pi_n^{\mc{W}} = \mathrm{Id}_{L^2(\RR^d)}$, while $\Pi_n^{\mc{W}}$ is not projection due to the redundancy of windowed Bloch transform.

\section{Formulation and main results}\label{sec:main_results}


Let us start with fixing some more notations. We will switch between
physical domain and phase space in the FGA formulation. For clarity,
we will use $\bd{x}, \bd{y}\in\mathbb{R}^{d}$ as spatial variables,
$(\bd{q},\bd{p})\in\mathbb{R}^{2d}$ as phase space variables. The
capital letters $\bd{X}$ and $\bd{Y}$ are shorthand notations for
$\bd{X}=\bd{x}/\veps$ and $\bd{Y}=\bd{y}/\veps$.

We define an effective (classical) Hamiltonian corresponding to each
energy band by
\begin{equation}
  h_{n}(\bd{q},\bd{p})=E_{n}(\bd{p})+U(\bd{q}).
\end{equation}
The associated Hamiltonian flow $\kappa_{n}(t)=(\bd{Q}_{n}(t,\bd{q},\bd{p}),\bd{P}_{n}(t,\bd{q},\bd{p}))$
solves
\begin{equation}\label{eq:Hamflow}
  \begin{cases}
  \displaystyle \dfrac{d\bd{Q}_{n}}{dt}=\nabla_{\bd{P}_{n}} h_n(\bd{Q_{n}},\bd{P}_{n}), \\[1em]
  \displaystyle \dfrac{d\bd{P}_{n}}{dt}=-\nabla_{\bd{Q}_{n}}
  \bd{Q}_{n},\bd{P}_{n})
  \end{cases}
\end{equation}
on $\Omega$ with initial conditions
$\bd{Q}_{n}(0,\bd{q},\bd{p})=\bd{q}$ and
$\bd{P}_{n}(0,\bd{q},\bd{p})=\bd{p}$.

From now on, we will use the short hand notation
$(\bd{Q}_{n},\bd{P}_{n})$ for
$(\bd{Q}_{n}(t,\bd{q},\bd{p}),\bd{P}_{n}(t,\bd{q},\bd{p}))$. For the
long time existence of the Hamiltonian flow \eqref{eq:Hamflow}, we
will assume that the external potential $U(\bd{x})$ is subquadratic as
below.
\begin{defn}\label{subquad} A potential $U$ is called \emph{subquadratic}, if
  $\norm{\partial_{\bd{x}}^{\alpha} U(\bd{x})}_{L^{\infty}}$ is finite
  for all multi-index $|\alpha|\geq 2$.
\end{defn}

\begin{remark}
  As a result, since the domain $\Gamma^{\ast}$ for $\bd{p}$ is
  bounded, the Hamiltonian $h_n$ is also subquadratic.
\end{remark}

The frozen Gaussian approximation will be formulated by the following Fourier
integral operator.
\begin{defn}(Fourier Integral Operator) For
  $u\in\mathcal{S}(\mathbb{R}^{2d}\times\Omega,\mathbb{C})$ and
  $\varphi\in \mathcal{S}(\mathbb{R}^{d},\mathbb{C})$ we define the
  Fourier Integral Operator with symbol $u$ by the oscillatory
  integral
  \begin{equation}
    [\mathcal{I}^{\veps}(u)\varphi](\bd{x}) =\dfrac{1}{(2\pi\veps)^{3d/2}}
    \iint_{\Omega}\int_{\mathbb{R}^{d}}
    e^{\frac{\I}{\veps}\Phi(t,\bd{x},\bd{y},\bd{q},\bd{p})}
    u(\bd{x},\bd{y},\bd{q},\bd{p})\varphi(\bd{y})
    \ud\bd{y}\ud\bd{q}\ud\bd{p}
  \end{equation}
  where the complex valued phase function $\Phi(t,\bd{x},\bd{y},\bd{q},\bd{p})$ is given by
  \begin{equation}
    \Phi(t,\bd{x},\bd{y},\bd{q},\bd{p})=S(t,\bd{q}, \bd{p})-\bd{p}\cdot (\bd{y}-\bd{q})+\bd{P}\cdot(\bd{x}-\bd{Q})+\dfrac{\I}{2}|\bd{y}-\bd{q}|^{2}+\dfrac{\I}{2}|\bd{x}-\bd{Q}|^{2}
  \end{equation}
  and $S(t,\bd{q},\bd{p})$ is a real-valued \emph{action} function associated
  to $\kappa$ satisfying
  \begin{equation}
    \nabla_{\bd{q}} S(t,\bd{q},\bd{p})=-\bd{p}+\nabla_{\bd{q}} \bd{Q}\cdot\bd{P}, \qquad
    \nabla_{\bd{p}}S(t,\bd{q},\bd{p})=\nabla_{\bd{p}} \bd{Q}\cdot\bd{P}.
  \end{equation}
\end{defn}
Note that if $\kappa(t) = \kappa_n(t)$, the action
$S_{n}(t,\bd{q},\bd{p})$ can be obtained by solving the
evolution equation
\begin{equation}
\dfrac{dS_{n}}{dt}=\bd{P}_{n}\cdot\nabla_{\bd{P}_{n}} h_{n}(\bd{Q}_{n},\bd{P}_{n})-h_{n}(\bd{Q}_{n},\bd{P}_{n}),
\end{equation}
with initial condition $S_{n}(0,\bd{q},\bd{p})=0$.

\smallskip

We are now ready to formulate the frozen Gaussian approximation.  The
FGA approximates the solution of the Schr\"odinger equation
\eqref{eq:schrodinger} on the $n$-th band to the leading order by
\begin{equation}
  \psi_{\FGA}^{\veps}(t, \bd{x}) = \left[\mathcal{I}^{\veps}\left(a_{n,0}(t,\bd{q},\bd{p})u_{n}(\bd{P}_{n},\dfrac{\bd{x}}{\veps})
      \wb{u}_{n}(\bd{p},\dfrac{\bd{y}}{\veps})\right)\psi_{0}^{\veps}\right](\bd{x}),
\end{equation}
where $\psi_0^{\veps}$ is the initial condition.  More explicitly, at
time $t$, $\psi_{\FGA}^{\veps}$ is given by
\begin{multline}\label{eq:FGAgaussian}
  \psi_{\FGA}^{\veps}(t, \bd{x}) = \frac{1}{(2\pi \veps)^{3d/2}}
  \iint_{\Omega}  a_{n, 0}(t, \bd{q}, \bd{p}) e^{\I S_n (t,\bd{q},\bd{p}) / \veps}
  G^{\veps}_{\bd{Q}_n, \bd{P}_n}(\bd{x}) u_n(\bd{P}_n, \bd{x}/ \veps) \\
\cdot \average{G^{\veps}_{\bd{q}, \bd{p}} u_n(\bd{p}, \cdot / \veps), \psi_0^{\veps}} \ud \bd{q} \ud \bd{p}.
\end{multline}
Here and in the sequel, we use the short-hand notation for
Gaussians with semiclassical scaling
\begin{equation}\label{eq:rescaledG}
  G_{\bd{q},\bd{p}}^{\veps}(\bd{x}):=\exp\left(- \frac{|\bd{x}-\bd{q}|^{2}}{2\veps} +\I \frac{\bd{p}\cdot(\bd{x}-\bd{q})}{\veps}\right),
\end{equation}
where the subscripts $(\bd{q}, \bd{p})$ indicate the center of the
Gaussian in phase space. Note that the semiclassical Fourier transform of
$G_{\bd{q}, \bd{p}}^{\veps}$ is
\begin{equation}
  \wh{G}_{\bd{q},\bd{p}}^{\veps}(\bd{\xi})=
  \frac{1}{(2 \pi \veps)^{d/2}} \int_{\RR^d}
  G_{\bd{q}, \bd{p}}^{\veps}(\bd{x}) e^{-\I \bd{\xi} \cdot \bd{x} / \veps} \ud \bd{x}
  = \exp\left(-\frac{|\bd{\xi}-\bd{p}|^{2}}{2\veps}
    +\I\frac{\bd{q}\cdot(\bd{\xi}-\bd{p})}{\veps}\right).
\end{equation}
Correspondingly, the semiclassical windowed Bloch transform $\mc{W}^{\veps}:
L^2(\RR^d) \to L^2(\Omega)^{\NN}$ is defined as
\begin{equation}
  (\mc{W}^{\veps} f)_{n}(\bd{q},\bd{p}) =  \frac{2^{d/4}}{(2\pi\veps)^{3d/4}} \braket{u_n(\bd{p}, \cdot / \veps) G^{\veps}_{\bd{q}, \bd{p}}, f} = \frac{2^{d/4}}{(2\pi\veps)^{3d/4}} \int_{\RR^d} \wb{u}_n(\bd{p}, \bd{x}/\veps)  \wb{G}^{\veps}_{\bd{q},\bd{p}}(\bd{x}) f(\bd{x}) \ud \bd{x}.
\end{equation}
Similarly we also have the  operator $\Pi_n^{\mc{W}, \veps}:
L^2(\RR^d) \to L^2(\RR^d)$ for each $n \in \NN$ with semiclassical scaling
\begin{equation}
  (\Pi_n^{\mc{W}, \veps} f)(\bd{y}) = \frac{2^{d/4}}{(2\pi\veps)^{3d/4}}
  \iint_{\Omega} u_n(\bd{\xi}, \bd{y}/\veps) G^{\veps}_{\bd{x},\bd{\xi}}(\bd{y})
  (\mc{W}^{\veps} f)_{n}(\bd{x},\bd{\xi})  \ud\bd{x}\ud\bd{\xi}.
\end{equation}
It follows from \eqref{eq:reconstruction} and a change of variable
that $\sum_n \Pi_n^{\mc{W}, \veps} = \mathrm{Id}_{L^2(\RR^d)}$.

The only term in \eqref{eq:FGAgaussian} that remains to be specified is the amplitude
$a_{n,0}(t,\bd{q},\bd{p})$. It solves the evolution equation
\begin{multline}
  \partial_{t} a_{n, 0}=-\I a_{n, 0}
  \mathcal{A}_n(\bd{P}_{n}) \cdot\nabla
  U(\bd{Q}_{n}) +
  \dfrac{1}{2} a_{n, 0} \tr \left(\partial_{\bd{z}}\bd{P}_{n}\,\nabla^2 E(\bd{P}_{n}) \bigl(Z_{n}\bigr)^{-1}\right)\\
  -\dfrac{\I}{2} a_{n, 0} \tr
  \left(\partial_{\bd{z}}\bd{Q}_{n} \, \nabla^2
    U(\bd{Q}_{n})
    \bigl(Z_{n}\bigr)^{-1}\right),
\end{multline}
with initial conditions $a_{n,0}(0,\bd{q},\bd{p})=2^{d/2}$ for each
$(\bd{q},\bd{p})$ and we recall that
$\mathcal{A}_n(\bd{\xi})=\braket{u_{n}(\bd{\xi},\cdot),\I\nabla_{\bd{\xi}}u_{n}(\bd{\xi},\cdot)}$
is the Berry phase.  Here the matrix $Z$ associated with the
Hamiltonian flow $\kappa_n(t)$ is defined by
\begin{equation}
  Z_{n}(t,q,p):=\partial_{\bd{z}}\left(\bd{Q}_{n} +\I\bd{P}_{n}\right),
\end{equation}
where $\partial_{\bd{z}}:=\partial_{\bd{q}}-\I\partial_{\bd{p}}$.


We now state the main results of this work.
\begin{theorem}\label{theoremA}
  Assume that the $n$-th Bloch band $E_{n}(\bd{\xi})$ does not
  intersect any other Bloch bands for all $\bd{\xi}\in \Gamma^{\ast}$
  and the Hamiltonian $h_n(\bd{x}, \bd{\xi})$ is subquadratic. Let
  $\mathscr{U}^{\veps}_t$ be the propagator of the time-dependent
  Schr\"{o}dinger equation~\eqref{eq:schrodinger} with initial condition $\psi_{0}^{\veps}\in L^{2}(\mathbb{R}^{d})$.
  Then for any given $T$, $0 \leq t \leq T$ and sufficiently small
  $\veps$, we have
  \begin{equation}\label{eq:theoremA}
    \sup_{0\leq t\leq T} \, \Bigl\lVert\, \mathscr{U}^{\veps}_t \bigl( \Pi^{\mc{W}, \veps}_n \psi_0^{\veps}\bigr) -
      \mathcal{I}^{\veps}\bigl(
        a_{n,0}u_{n}(\bd{P}_{n},\bd{x}/\veps)
        \wb{u}_{n}(\bd{p},\bd{y}/\veps)\bigr)
    \psi_{0}^{\veps} \, \Bigr\rVert_{L^{2}}
    \leq C_{T, n} \, \veps\, \bigl\lVert \psi_{0}^{\veps} \bigr\rVert_{L^2}.
  \end{equation}
\end{theorem}

\begin{remark}
  Note that the FGA solution approximates the time evolution of
  $\Pi_n^{\mc{W}, \veps} \psi_0^{\veps}$, which is the $n$-th band
  contribution to the initial condition in the reconstruction formulae
  \eqref{eq:reconstruction}. In particular, if the initial condition is
  concentrated on the $n$-th band in the sense that $\psi_{0}^{\veps}
  = \Pi^{\mc{W}, \veps}_{n}\psi_{0}^{\veps}$, the theorem states
  that the solution to \eqref{eq:schrodinger} is approximated by the
  FGA solution with $\Or(\veps)$ error.
\end{remark}

\begin{remark}
  We can also construct higher order approximations by replacing the
  term $a_{n,0}u_{n}(\bd{P}_{n},\bd{x}/\veps)$ with an
  $\veps$-expansion of the form $b_{n,0} +\veps
  b_{n,1}+\veps^{2}b_{n,2}+\ldots + \veps^{N-1} b_{n, N-1}$ where
  $b_{n,0} = a_{n,0}u_{n}(\bd{P}_{n},\bd{x}/\veps)$. This will give an
  approximate solution $\psi_{\FGA}^{\veps,N}$ to $\Or(\veps^N)$
  accuracy.  In this paper we shall focus on the first order
  approximation and omit the formulation and proof for higher orders.
\end{remark}

\begin{remark}
  Let us also remark that while we take the more explicit approach of
  using Bloch waves in a modified FGA ansatz for periodic media, as in
  \eqref{eq:FGAgaussian}. The same approximation can be also derived
  by first projecting the whole Schr\"odinger equation using a
  super-adiabatic projection as developed in
  \cite{PaSpTe:03,PaSpTe:06} and then apply the frozen Gaussian
  approximation to the resulting dynamics. We will not go into the
  details in this work.
\end{remark}

The proof of Theorem~\ref{theoremA} is given in Section
\ref{sec:proof}. By linearity of \eqref{eq:schrodinger}, we have the
following more general statement, as an easy corollary from
Theorem~\ref{theoremA}.
\begin{theorem}\label{theoremB}
  Assume that the first $N$ Bloch bands $E_{n}(\bd{\xi})$,
  $n=1,\cdots,N$ do not intersect and are separated from the other
  bands for all $\bd{\xi}\in \Gamma^{\ast}$; and assume that the
  Hamiltonian $h_n(\bd{x}, \bd{\xi})$ is subquadratic. Let
  $\mathscr{U}^{\veps}_t$ be the propagator of the time-dependent
  Schr\"{o}dinger equation~\eqref{eq:schrodinger} with initial condition $\psi_{0}^{\veps}\in L^{2}(\mathbb{R}^{d})$. Then for any given
  $T$, $0 \leq t \leq T$ and sufficiently small $\veps$, we have
  \begin{multline}
    \sup_{0\leq t\leq T} \, \biggl\lVert\,\mathscr{U}^{\veps}_t
    \psi_0^{\veps} -
    \sum_{n=1}^{N}\mathcal{I}^{\veps}\bigl(
    a_{n,0}u_{n}(\bd{P}_{n},\bd{x}/\veps)
    \wb{u}_{n}(\bd{p},\bd{y}/\veps)\bigr)
    \psi_{0}^{\veps} \, \biggr\rVert_{L^{2}} \\
    \leq C_{T, N}\, \veps\bigl\lVert
    \psi_{0}^{\veps}\bigr\rVert_{L^2}+ \norm{ \psi_0^{\veps} -
      \sum_{n=1}^N \Pi^{\mc{W}, \veps}_n \psi_0^{\veps}}_{L^2}.
  \end{multline}
\end{theorem}

\begin{proof}
  Taking the short-hand notation $\psi_{0, n}^{\veps} = \Pi^{\mc{W},
    \veps}_n \psi_0^{\veps}$ and $\mathscr{V}_{t,n}^{\veps} =
  \mathcal{I}^{\veps}\bigl(a_{n,0}
  u_{n}(\bd{P}_{n},\bd{x}/\veps)
  \wb{u}_{n}(\bd{p},\bd{y}/\veps)\bigr)$, we have
  \begin{align*}
    \norm{\mathscr{U}^{\veps}_t\psi_{0}^{\veps}-\sum_{n=1}^{N}\mathscr{V}_{t,n}^{\veps}\psi_{0}^{\veps}}_{L^{2}}&=
    \norm{\mathscr{U}^{\veps}_t\left(\sum_{n=1}^{\infty}\psi_{0,n}^{\veps}\right)-\sum_{n=1}^{N}\mathscr{V}_{t,n}^{\veps}\psi_{0}^{\veps}}_{L^{2}}\\
    &=\norm{\mathscr{U}^{\veps}_t\left(\sum_{n=1}^{N}\psi_{0,n}^{\veps}\right)+\mathscr{U}^{\veps}_t\left(\sum_{n=N+1}^{\infty}\psi_{0,n}^{\veps}\right)-\sum_{n=1}^{N}\mathscr{V}_{t,n}^{\veps}\psi_{0}^{\veps}}_{L^{2}}\\
    &\leq\norm{\mathscr{U}^{\veps}_t\left(\sum_{n=1}^{N}\psi_{0,n}^{\veps}\right)-\sum_{n=1}^{N}\mathscr{V}_{t,n}^{\veps}\psi_{0}^{\veps}}_{L^{2}}+
    \norm{\mathscr{U}^{\veps}_t\left(\sum_{n=N+1}^{\infty}\psi_{0,n}^{\veps}\right)}_{L^{2}}\\
    &\stackrel{\eqref{eq:theoremA}}{\leq} \sum_{n=1}^{N}C_{T, n} \veps
    \norm{\psi_{n,0}^{\veps}}_{L^{2}}+
    \norm{\sum_{n=N+1}^{\infty}\psi_{0}^{\veps}}_{L^{2}} \\
    &\leq C_{T, N}\veps\norm{\psi_{0}^{\veps}}_{L^{2}} + \norm{
      \psi_0^{\veps} - \sum_{n=1}^N \Pi^{\mc{W}, \veps}_n
      \psi_0^{\veps}}_{L^2}.
\end{align*}
\end{proof}

\section{Analysis of frozen Gaussian approximation in periodic media}\label{sec:proof}

\subsection{Initial condition}\label{ICsection}

Let us first study the initial condition for the frozen Gaussian approximation. At time $t = 0$, observe that by setting $t = 0$ in
\eqref{eq:FGAgaussian} we have
\begin{equation*}
    \psi^{\veps}_{FGA, n}(0,\bd{x}) = \frac{2^{d/2}}{(2\pi
      \veps)^{3d/2}} 
    \iint_{\Omega} u_n\bigl(\bd{p}, {\bd{x}}/{\veps}\bigr) G^{\veps}_{\bd{q},
      \bd{p}}(\bd{x}) \average{ G^{\veps}_{\bd{q}, \bd{p}} u_n(\bd{p}, \cdot /
      \veps), \psi_0^{\veps}}  \ud \bd{q} \ud \bd{p} = \Pi^{\mc{W}, \veps}_n \psi_0^{\veps}
\end{equation*}
by definition of the operator $\Pi^{\mc{W}, \veps}_n$. Hence, the FGA solution matches $\Pi^{\mc{W}, \veps}_n \psi_0^{\veps}$ at $t = 0$.

\subsection{Estimates of the Hamiltonian flows}

To control the error for $t > 0$, we collect here some preliminary
results on the estimate of quantities associated with the Hamiltonian
flows. We will assume throughout the rest of the paper that the
assumptions of Theorem~\ref{theoremA} hold for a fixed Bloch band $n$.

The following notation is useful in the proof.  For
$u\in\mathcal{C}^{\infty}(\Omega,\mathbb{C})$, we define for
$k\in\mathbb{N}$,
\begin{equation}
M_{k}[u]=\displaystyle\max_{|\alpha_{q}|+|\alpha_{p}|\leq k}\displaystyle\sup_{(\bd{q},\bd{p})\in\Omega}\left|\partial_{\bd{q}}^{\alpha_{q}}\partial_{\bd{p}}^{\alpha_{p}}u(\bd{q},\bd{p})\right|
\end{equation}
where $\alpha_{q}$ and $\alpha_{p}$ are multi-indices corresponding to
$\bd{q}$ and $\bd{p}$, respectively.

\begin{defn}(Canonical Transformation) Let
  $\kappa:\mathbb{R}^{2d}\rightarrow \mathbb{R}^{2d}$ be a
  differentiable map $\kappa(\bd{q},\bd{p}) =
  (\bd{Q}(\bd{q},\bd{p}), \bd{P}(\bd{q},\bd{p}))$
  and denote the Jacobian matrix as
  \begin{equation}
    (F)=\begin{pmatrix}
      (\partial_{\bd{q}}\bd{Q})^{T}(\bd{q},\bd{p}) & (\partial_{\bd{p}}\bd{Q})^{T}(\bd{q},\bd{p})\\
      (\partial_{\bd{q}}\bd{P})^{T}(\bd{q},\bd{p}) & (\partial_{\bd{p}}\bd{P})^{T}(\bd{q},\bd{p})
    \end{pmatrix}.
  \end{equation}
  We say $\kappa$ is a \emph{canonical transformation} if $F$
  is symplectic for any $(\bd{q},\bd{p})\in \mathbb{R}^{2d}$, i.e.
  \begin{equation}
    \left(F\right)^{T} \begin{pmatrix}
      0 & \Id_{d}\\
      -\Id_{d} & 0
    \end{pmatrix}
    F= \begin{pmatrix}
      0 & \Id_{d}\\
      -\Id_{d} & 0
    \end{pmatrix}.
  \end{equation}
\end{defn}

It is easy to check by the definition that the map
$\kappa_{n}(t):\mathbb{R}^{2d}\rightarrow \mathbb{R}^{2d}$ defined by
$(\bd{q},\bd{p})\rightarrow (\bd{Q}_{n}(t,\bd{q},\bd{p}),
\bd{P}_{n}(t,\bd{q},\bd{p}))$ solving \eqref{eq:Hamflow} is a
canonical transformation.

\begin{prop}\label{kappabounds}
  We have for all $k\geq 0$
  \begin{equation}\label{eq:kappaassumptions}
    \displaystyle\sup_{t\in [0,T]}M_{k}\left[F_{n}(t)\right]<\infty \hspace{1cm}\displaystyle\sup_{t\in [0,T]}M_{k}\left[\frac{\ud}{\ud t}F_{n}(t)\right]<\infty.
  \end{equation}
\end{prop}

\begin{proof}
  Differentiating $F_{n}(t,\bd{q},\bd{p})$ with respect to $t$ gives
  \begin{equation}\label{eq:Fderivatives}
    \dfrac{\ud}{\ud t}F_{n}(t,\bd{q},\bd{p})=\begin{pmatrix}
      \partial_{\bd{P}}\partial_{\bd{Q}}h_n & \partial_{\bd{P}}\partial_{\bd{P}}h_n\\
      -\partial_{\bd{Q}}\partial_{\bd{Q}}h_n & -\partial_{\bd{Q}}\partial_{\bd{P}}h_n
    \end{pmatrix}F_{n}(t,\bd{q},\bd{p}).
  \end{equation}
  By our assumption that $U$ is subquadratic on $\mathbb{R}^{d}$ and
  since $E_n \in \mathcal{C}^{\infty}(\Gamma^{\ast})$, there exists a
  constant $C$ independent of $(\bd{q},\bd{p})$ such that
\begin{equation}
\frac{\ud}{\ud t}|F_{n}(t,\bd{q},\bd{p})|=\left\lvert\begin{pmatrix}
\partial_{\bd{P}}\partial_{\bd{Q}}h_n & \partial_{\bd{P}}\partial_{\bd{P}}h_n\\
-\partial_{\bd{Q}}\partial_{\bd{Q}}h_n & -\partial_{\bd{Q}}\partial_{\bd{P}}h_n
\end{pmatrix}\right\rvert \bigl\lvert
F_{n}(t,\bd{q},\bd{p})\bigr\rvert \leq
C \bigl\lvert
F_{n}(t,\bd{q},\bd{p})\bigr\rvert
\end{equation}
with $\abs{F_{n}(0)}=\abs{\Id_{2d}}$. By an application of
Gronwall's inequality, we obtain
\begin{equation}
  |F_{n}(t)|\leq e^{C\abs{t}}.
\end{equation}
Differentiating \eqref{eq:Fderivatives} with respect to $(\bd{q},\bd{p})$ yields
\begin{multline}
\dfrac{\ud}{\ud t}\partial_{\bd{q}}^{\alpha_{\bd{q}}}\partial_{\bd{p}}^{\alpha_{\bd{p}}}F_{n}(t,\bd{q},\bd{p}) = \sum_{\beta_{\bd{q}}\leq\alpha_{\bd{q}},\beta_{\bd{p}}\leq\alpha_{\bd{p}}}
{\alpha_{\bd{q}} \choose \beta_{\bd{q}}}
{\alpha_{\bd{p}} \choose \beta_{\bd{p}}}
\partial_{\bd{q}}^{\beta_{\bd{q}}}\partial_{\bd{p}}^{\beta_{\bd{p}}}
\begin{pmatrix}
  \partial_{\bd{P}}\partial_{\bd{Q}} h_n & \partial_{\bd{P}}\partial_{\bd{P}} h_n \\
  -\partial_{\bd{Q}}\partial_{\bd{Q}} h_n &
  -\partial_{\bd{Q}}\partial_{\bd{P}} h_n
\end{pmatrix} \times \\
\times\partial_{\bd{q}}^{\alpha_{\bd{q}}-\beta_{\bd{q}}}\partial_{\bd{p}}^{\alpha_{\bd{p}}-\beta_{\bd{p}}}F_{n}(t,\bd{q},\bd{p}).
\end{multline}
Our estimate now follows by induction.
\end{proof}

Recall that the matrix $Z_{n}(t,\bd{q},\bd{p})$ is defined by
\begin{equation}
  Z_{n}(t,\bd{q},\bd{p}):= \partial_{\bd{z}}\left(\bd{Q}_{n}(t,\bd{q},\bd{p})+\I\bd{P}_{n}(t,\bd{q},\bd{p})\right)
  = (\partial_{\bd{q}}-\I\partial_{\bd{p}}) \left(\bd{Q}_{n}(t,\bd{q},\bd{p})+\I\bd{P}_{n}(t,\bd{q},\bd{p})\right).
\end{equation}
We have the following. It follows the same proof of \cite{LuYa:11}*{Proposition 3.5}, which we reproduce here for completeness.
\begin{prop}\label{Zinvbound}
  $Z_{n}(t,\bd{q},\bd{p})$ is invertible for
  $(\bd{q},\bd{p})\in\Omega$. Moreover, for each $k\in\mathbb{N}$,
  \begin{equation}
    M_{k}\Bigl[\bigl(Z_{n}(t)\bigr)^{-1}\Bigr]<\infty.
  \end{equation}
\end{prop}

\begin{proof}
  $Z_{n}(t,\bd{q},\bd{p})$ inherits the property that
  $M_{k}(Z_{n}(t,\bd{q},\bd{p}))<\infty$ from the same
  estimate for $F_{n}(t,\bd{q},\bd{p})$. Moreover, we have
  \begin{equation}
    \begin{split}
      Z_{n}(Z_{n})^{*}(t,\bd{q},\bd{p})=&\begin{pmatrix}
       \I\Id_{d} & \Id_{d}
      \end{pmatrix}(F_{n})^{T}(t,\bd{q},\bd{p})\begin{pmatrix}
        \Id_{d} & -\I\Id_{d}\\
        \I\Id_{d} & \Id_{d}
      \end{pmatrix}F_{n}(t,\bd{q},\bd{p})\begin{pmatrix}
        -\I\Id_{d}\\
        \Id_{d}
      \end{pmatrix}\\
      =&\begin{pmatrix}
        \I\Id_{d} & \Id_{d}
      \end{pmatrix}
      \left((F_{n})^{T}(F_{n})\right)(t,\bd{q},\bd{p})
      \begin{pmatrix}
        -\I\Id_{d}\\
        \Id_{d}
      \end{pmatrix}\\
      &+\begin{pmatrix}
        \I\Id_{d} & \Id_{d}
      \end{pmatrix}(F_{n})^{T}(t,\bd{q},\bd{p})\begin{pmatrix}
        0 & -\I\Id_{d}\\
        \I\Id_{d} & 0
      \end{pmatrix}F_{n}(t,\bd{q},\bd{p})\begin{pmatrix}
        -\I\Id_{d}\\
        \Id_{d}
      \end{pmatrix}\\
      =&\begin{pmatrix}
        \I\Id_{d} & \Id_{d}
      \end{pmatrix}\left((F_{n})^{T}F_{n}\right)(t,\bd{q},\bd{p})\begin{pmatrix}
        -\I\Id_{d}\\
        \Id_{d}
      \end{pmatrix}+2\Id_{d}.
    \end{split}
  \end{equation}

  This calculation shows that, since
  $(F_{n}(t))^{T}F_{n}(t)$ is semi-positive
  definite, for any $\bd{v} \in \CC^{2d}$,
  \begin{equation}
    \bd{v}^{*}Z_{n}(t)(Z_{n}(t))^{*}\bd{v}\geq 2|\bd{v}|^{2}.
  \end{equation}
  Therefore $Z_{n}(t,\bd{q},\bd{p})$ is invertible and
  $\det \bigl(Z_{n}(t)\bigr)$ is uniformly bounded away from
  $0$ for all $\bd{q}$ and $\bd{p}$, so by representing
  $(Z_{n})^{-1}(t,\bd{q},\bd{p})$ by minors,
  $M_{k}\bigl((Z_{n})^{-1}(t,\bd{q},\bd{p})\bigr)<\infty$,
  as $M_{k}(Z_{n}(t,\bd{q},\bd{p}))$ is.
\end{proof}

\begin{prop}\label{ubounds}
  For each $k\in\mathbb{N}$,
  \begin{equation}
    \displaystyle\sup_{t\in[0,T]}M_{k}\left[u_{n}(\bd{P}_{n},\bd{x})\right]<\infty.
  \end{equation}
\end{prop}

\begin{proof}
  $u_{n}(\bd{P}_{n},\bd{x})$ is smooth on the compact set
  $\Gamma^{\ast} \times \Gamma$ since the $n$-th band is separated
  from the rest of the spectrum (see e.g., \cite{ReedSimon4}*{Sec XIII.16}). Thus
  $u_{n}(\bd{P}_{n},\bd{x})$ is uniformly bounded on
  $\Gamma^{\ast} \times \Gamma$ and hence $\Gamma^{\ast} \times \RR^d$
  due to periodicity. We also see from Proposition~\ref{kappabounds}
  that the derivatives of $u_{n}(\bd{P}_{n},\bd{x})$ are also
  bounded. Thus, $M_{k}[u_{n}(\bd{P}_{n},\bd{x})]<\infty$ for any
  finite time $t$.
\end{proof}

\subsection{Higher order asymptotic solution}

To prove the theorem, we will need to construct a solution to the
Schr\"odinger equation that is accurate up to $\Or(\veps^2)$.
The construction is based on matched asymptotic expansion.
Let us fix a band $n \in \mathbb{N}$ and consider the ansatz
\begin{equation}\label{eq:ansatz}
  \psi_{\FGA, \infty}^{\veps} = \dfrac{1}{(2\pi \veps)^{3d/2}}
  \iint_{\Omega}b^{\veps}(t,\bd{X},\bd{q},\bd{p})
  G^{\veps}_{\bd{Q}_{n},\bd{P}_{n}}
  e^{\I S_{n}/\veps}\braket{G_{\bd{q},\bd{p}}^{\veps}u_{n}(\bd{p},
    \cdot/\veps),\psi_{0}} \ud\bd{q}\ud\bd{p},
\end{equation}
where the coefficient $b$ assumes the asymptotic expansion
\begin{equation}
  \begin{aligned}
    b^{\veps}(t,\bd{X},\bd{q},\bd{p})
    & :=  \sum_{j=0}^{\infty} \veps^{j}b_{j}(t,\bd{X},\bd{q},\bd{p}) \\
    & =  a_{n,0}(t,\bd{q},\bd{p}) u_{n}(\bd{P}_{n},\bd{X}) \\
    & \qquad + \veps \bigl( a_{n,1}(t,\bd{q},\bd{p}) u_{n}(\bd{P}_{n},\bd{X})+b_{n,1}^{\perp}(t, \bd{X}, \bd{q}, \bd{p}) \bigr) \\
    & \qquad + \veps^2 \bigl( a_{n,2}(t,\bd{q},\bd{p}) u_{n}(\bd{P}_{n},\bd{X})+b_{n,2}^{\perp}(t, \bd{X}, \bd{q}, \bd{p}) \bigr) +
    \sum_{j=3}^{\infty} \veps^{j}b_{j}(t,\bd{X},\bd{q},\bd{p})
  \end{aligned}
\end{equation}


To determine the terms in the expansion, we will make use of the
following Lemma.
\begin{defn}
For $f=f(t,\bd{x},\bd{y},\bd{q},\bd{p})$ and $g=g(t,\bd{x},\bd{y},\bd{q},\bd{p})$ such that for any $t$ and $\bd{x}$,
\begin{equation*}
  f(t,\bd{x}, \cdot,\cdot,\cdot), g(t,\bd{x},\cdot,\cdot,\cdot)\in L^{\infty}(\mathbb{R}^{d};\mathcal{S}(\mathbb{R}^{d}\times\Gamma^{\ast})),
\end{equation*}
we say that $f$ and $g$ are equivalent for the $n$-th Bloch band, denoted as $f\sim_{n}g$ if for any $t\geq 0$ and $\Psi_{0}\in L^{2}(\mathbb{R}^{d})$
\begin{equation}
\iint_{\Omega}\int_{\RR^d} (f-g)(t,\bd{x},\bd{y},\bd{q},\bd{p})G^{\veps}_{\bd{Q}_{n},\bd{P}_{n}}e^{\I S_{n}(t,\bd{q},\bd{p})/\veps}
\wb{G}^{\veps}_{\bd{q},\bd{p}}(\bd{y})\Psi_{0}(\bd{y})
\ud\bd{y}\ud\bd{q}\ud\bd{p}=0.
\end{equation}
\end{defn}

\begin{lemma}\label{xtoepsilon}
For any $d$-vector function $\bd{v}(\bd{y},\bd{q},\bd{p})$ such that each component is in $L^{\infty}(\mathbb{R}^{d};\mathcal{S}(\mathbb{R}^{d}\times\Gamma^{\ast}))$
\begin{equation}
\bd{v}(\bd{y},\bd{q},\bd{p})\cdot(\bd{x}-\bd{Q}_{n})\sim_{n}-\veps\partial_{\bd{z}}\cdot(\bd{v}Z_{n}^{-1}),
\end{equation}
and for any $d\times d$ matrix function $M(\bd{y},\bd{q},\bd{p})$ such that each component is in $L^{\infty}(\mathbb{R}^{d};\mathcal{S}(\mathbb{R}^{d}\times\Gamma^{\ast}))$
\begin{equation}
\begin{split}
\tr\left(M(\bd{y},\bd{q},\bd{p})(\bd{x}-\bd{Q}_{n})^{2}\right)\sim_{n}&\veps\,\tr\left(\partial_{\bd{z}}\bd{Q}_{n}MZ_{n}^{-1}\right)-\veps\tr\left(\partial_{\bd{z}}M(\bd{x}-\bd{Q}_{n})Z_{n}^{-1}+ M(\bd{x}-\bd{Q}_{n})\partial_{\bd{z}}Z_{n}^{-1}\right)\\
=&\veps\tr\left(\partial_{\bd{z}}\bd{Q}_{n}MZ_{n}^{-1}\right)+\veps^{2}\tr\left(\partial_{\bd{z}}\left(\partial_{\bd{z}}M(Z_{n}^{-1})^{2}\right)+\partial_{\bd{z}}\left( M\partial_{\bd{z}}Z_{n}^{-1}\right)Z_{n}^{-1}\right).
\end{split}
\end{equation}
Higher order terms can be obtained recursively. In general we have for any multi-index $\alpha$ that $|\alpha|\geq 3$,
\begin{equation}
(\bd{x}-\bd{Q}_{n})^\alpha\sim_{n}\mathcal{O}\left(\veps^{\left\lfloor{\frac{|\alpha|+1}{2}}\right\rfloor}\right).
\end{equation}
\end{lemma}

\begin{proof}
  The proof of lemma \ref{xtoepsilon} is essentially the same as in Lemma 3 of
  \cite{SwRo:09}
and thus is omitted here.
\end{proof}

We now substitute \eqref{eq:ansatz} into the Schr\"odinger equation. For this we first compute the time and space derivatives on
$\psi_{\FGA, \infty}^{\veps}$:
\begin{align}
  \I\veps\partial_{t} \psi_{\FGA, \infty}^{\veps}=&\dfrac{1}{(2\pi
    \veps)^{3d/2}}\iint_{\Omega}\left\{i\veps\partial_{t}b^{\veps}
    -\left(\partial_{t}S_{n}-\bd{P}_{n}\cdot\partial_{t}\bd{Q}_{n}+(\partial_{t}\bd{P}_{n}-\I\partial_{t}\bd{Q}_{n})\cdot(\bd{x}-\bd{Q}_{n})\right)b^{\veps}\right\}\times \\
  & \hspace{10em} \notag \times G^{\veps}_{\bd{Q}_{n},\bd{P}_{n}}e^{\I S_{n}/\veps}
  \braket{G_{\bd{q},\bd{p}}^{\veps}u_{n}(\bd{p}, \cdot/\veps),\psi_{0}}\ud\bd{q}\ud\bd{p}. \\
  \tfrac{1}{2}\veps^{2}\Delta \psi_{\FGA, \infty}^{\veps}=&\dfrac{1}{(2\pi \veps)^{3d/2}}\iint_{\Omega}\left[-\dfrac{1}{2}(-\I\nabla_{\bd{X}}+\bd{P}_{n})^{2}b^{\veps}-(\nabla_{\bd{X}}b^{\veps}+\I b^{\veps}\bd{P}_{n})\cdot(\bd{x}-\bd{Q}_{n})\right. +\\
  &\hspace{10em} \qquad \notag\left.+ \dfrac{1}{2}b^{\veps}|\bd{x}-\bd{Q}_{n}|^{2}-\dfrac{1}{2}\veps
    b^{\veps} d\right]\times \\
  &\hspace{10em} \notag \times G^{\veps}_{\bd{Q}_{n},\bd{P}_{n}}e^{\I S_{n}/\veps}
  \braket{G_{\bd{q},\bd{p}}^{\veps}u_{n}(\bd{p},
    \cdot/\veps),\psi_{0}}\ud\bd{q}\ud\bd{p}.
\end{align}
Hence, after rearranging terms, we arrive at
\begin{equation}
\begin{split}
  \bigl(\I\veps\partial_{t} & +\dfrac{1}{2}\veps^{2}\Delta  -V(\bd{X})-U(\bd{x})\bigr)\psi_{\FGA, \infty}^{\veps} = \\
  & =\dfrac{1}{(2\pi \veps)^{3d/2}}\iint_{\Omega}
  \biggl\{\biggl[-\dfrac{1}{2}(-\I\nabla_{\bd{X}}+\bd{P}_{n})^{2}-V(\bd{X})
  - U(\bd{x})-\partial_{t}S_{n}\biggr]
  b^{\veps} + \\
  & \hspace{8em}
  +\veps\bigl(\I\partial_{t}b^{\veps}-\dfrac{1}{2}b^{\veps}d\bigr)
  -\left[(\nabla_{\bd{X}}b^{\veps}
    +\I b^{\veps}\bd{P}_{n})+(\partial_{t}\bd{P}_{n}
    -\I\partial_{t}\bd{Q}_{n})b^{\veps}\right]\cdot(\bd{x}-\bd{Q}_{n}) + \\
  & \hspace{8em} +\dfrac{1}{2}|\bd{x}-\bd{Q}_{n}|^{2}b^{\veps}
  +\bd{P}_{n}\cdot\partial_{t}\bd{Q}_{n}b^{\veps}\biggr\}G^{\veps}_{\bd{Q}_{n},\bd{P}_{n}}e^{\I S_{n}/\veps}
  \braket{G_{\bd{q},\bd{p}}^{\veps}u_{n}(\bd{p},
    \cdot/\veps),\psi_{0}}\ud\bd{q}\ud\bd{p}.
\end{split}
\end{equation}
Define
\begin{equation}
  \begin{aligned}
    f(t, \bd{x}, \bd{y}, \bd{q}, \bd{p}) & =
    \biggl\{\biggl[-\dfrac{1}{2}(-\I\nabla_{\bd{X}}+\bd{P}_{n})^{2}-V(\bd{X})
    - U(\bd{x})-\partial_{t}S_{n}\biggr]
    b^{\veps} + \\
    & \qquad
    +\veps \bigl(\I\partial_{t}b^{\veps}-\dfrac{1}{2}b^{\veps}d\bigr)
    -\left[(\nabla_{\bd{X}}b^{\veps}
      +\I b^{\veps}\bd{P}_{n})+(\partial_{t}\bd{P}_{n}
      -\I\partial_{t}\bd{Q}_{n})b^{\veps}\right]\cdot(\bd{x}-\bd{Q}_{n}) + \\
    & \qquad +\dfrac{1}{2}|\bd{x}-\bd{Q}_{n}|^{2}b^{\veps}
    +\bd{P}_{n}\cdot\partial_{t}\bd{Q}_{n}b^{\veps}\biggr\}
    \wb{u}_{n}(\bd{p},\bd{Y}),
  \end{aligned}
\end{equation}
then we can write
\begin{multline}
  \bigl(\I\veps\partial_{t} +\dfrac{1}{2}\veps^{2}\Delta  -V(\bd{X})-U(\bd{x})\bigr)\psi_{\FGA, \infty}^{\veps} = \\
  =
  \dfrac{1}{(2\pi \veps)^{3d/2}} \iint_{\Omega} \int_{\RR^d} f(t, \bd{x}, \bd{y}, \bd{p}, \bd{q}) G^{\veps}_{\bd{Q}_{n},\bd{P}_{n}}e^{\I S_{n}/\veps}\wb{G}_{\bd{q},\bd{p}}^{\veps}(\bd{y})\psi_{0}(\bd{y})
  \ud\bd{y}\ud\bd{q}\ud\bd{p}.
\end{multline}
Applying Lemma~\ref{xtoepsilon} and adding and subtracting
$U(\bd{Q}_{n})$, we get
\begin{equation}\label{eq:eqf}
\begin{split}
  f \sim_n
  & \Bigl(-\tfrac{1}{2}(-\I\nabla_{\bd{X}}+\bd{P}_{n})^{2}-V(\bd{X})-(U(\bd{x})-U(\bd{Q}_{n}))-\partial_{t}S_{n} \Bigr) b^{\veps}\wb{u}_{n}(\bd{p},\bd{Y})\\
  & \qquad +\veps
  \bigl(\I\partial_{t}b^{\veps}-\dfrac{1}{2}b^{\veps}d\bigr)
  \wb{u}_{n}(\bd{p},\bd{Y})\\
  & \qquad + \veps\partial_{\bd{z}} \Bigl(\left[(\nabla_{\bd{X}}
    b^{\veps} +\I
    b^{\veps}\bd{P}_{n})+(\partial_{t}\bd{P}_{n}-\I\partial_{t}\bd{Q}_{n})b^{\veps}\right]
  \wb{u}_{n}(\bd{p},\bd{Y})Z_{n}^{-1}\Bigr)\\
  & \qquad +\veps\dfrac{1}{2} b^{\veps} \tr\left[\partial_{\bd{z}}\bd{Q}_{n}Z_{n}^{-1}\right]\wb{u}_{n}(\bd{p},\bd{Y})+\veps^{2}\dfrac{1}{2}\tr\left[\partial_{\bd{z}}\left(\partial_{\bd{z}}\left(b^{\veps}\wb{u}_{n}(\bd{p},\bd{Y})Z_{n}^{-1}\right)Z_{n}^{-1}\right)\right]\\
  & \qquad
  +\bd{P}_{n}\cdot\partial_{t}\bd{Q}_{n}b^{\veps}\wb{u}_{n}(\bd{p},\bd{Y})-U(\bd{Q}_{n})b^{\veps}
  \wb{u}_{n}(\bd{p},\bd{Y})
\end{split}
\end{equation}
Use the Taylor expansion of $U(\bd{x})$ about $\bd{Q}_{n}$
\begin{multline}\label{eq:eq2}
  (U(\bd{x})-U(\bd{Q}_{n})) = \nabla U(\bd{Q}_{n})(\bd{x}-\bd{Q}_{n}) +
  \dfrac{1}{2!} \nabla^2 U(\bd{Q}_{n})(\bd{x}-\bd{Q}_{n})^{2}\\
  + \dfrac{1}{3!} \nabla^{3}U(\bd{Q}_{n})
    (\bd{x}-\bd{Q}_{n})^{3} +
  \dfrac{1}{4!} \nabla^{4}U(\bd{Q}_{n}) (\bd{x}-\bd{Q}_{n})^{4}
  +\sum_{|\alpha|=5}R_{\alpha}(\bd{x})(\bd{x}-\bd{Q}_{n})^{\alpha}
\end{multline}
with
\begin{equation}
R_{\alpha}(\bd{x})=\dfrac{|\alpha|}{5!}\displaystyle\int_{0}^{1}(1-\tau)^{|\alpha|-1}\partial_{\bd{Q}_{n}}^{\alpha}U(\bd{Q}_{n}+\tau(\bd{x}-\bd{Q}_{n}))d\tau.
\end{equation}
From now on, let us denote the remainder term in \eqref{eq:eq2} by $R(\bd{x},\bd{q},\bd{p})$.

Applying Lemma~\ref{xtoepsilon} again to \eqref{eq:eqf} together with \eqref{eq:eq2}, we obtain
\begin{equation}
\begin{split}
  f & \sim_n
  \Bigl(-\dfrac{1}{2}(-\I\nabla_{\bd{X}}+\bd{P}_{n})^{2}-V(\bd{X})-\partial_{t}S_{n}\Bigr)
  b^{\veps}\wb{u}_{n}(\bd{p},\bd{Y})\\
  &\qquad +\bd{P}_{n}\cdot\partial_{t}\bd{Q}_{n}b^{\veps}\wb{u}_{n}(\bd{p},\bd{Y})-U(\bd{Q}_{n})b^{\veps}\wb{u}_{n}(\bd{p},\bd{Y})\\
  &\qquad +\veps\Bigl(\I\partial_{t}b^{\veps}-\dfrac{1}{2}b^{\veps}d\Bigr)\wb{u}_{n}(\bd{p},\bd{Y})+\veps\partial_{\bd{z}}\left(\nabla U(\bd{Q}_{n})b^{\veps}\wb{u}_{n}(\bd{p},\bd{Y})Z_{n}^{-1}\right)\\
  &\qquad +\veps\partial_{\bd{z}}\left(\left[(\nabla_{\bd{X}}b^{\veps}+\I b^{\veps}\bd{P}_{n})+(\partial_{t}\bd{P}_{n}-\I\partial_{t}\bd{Q}_{n})b^{\veps}\right]\wb{u}_{n}(\bd{p},\bd{Y})Z_{n}^{-1}\right)\\
  &\qquad +\veps\dfrac{1}{2!}\tr\left[\partial_{\bd{z}}\bd{Q}_{n}(I-\nabla^2 U(\bd{Q}_{n}))b^{\veps}\wb{u}_{n}(\bd{p},\bd{Y})Z_{n}^{-1}\right]\\
  &\qquad +\veps^{2}\dfrac{1}{2!}\tr\left[\partial_{\bd{z}}(\partial_{\bd{z}}((I-\nabla^2 U(\bd{Q}_{n}))b^{\veps}\wb{u}_{n}(\bd{p},\bd{Y})Z_{n}^{-1})Z_{n}^{-1})\right]\\
  &\qquad +\veps^{2}\dfrac{2}{3!}\tr\left[\partial_{\bd{z}}(\partial_{\bd{z}}\bd{Q}_{n}\nabla^{3}U(\bd{Q}_{n})b^{\veps}\wb{u}_{n}(\bd{p},\bd{Y})(Z_{n}^{-1})^{2})\right]\\
  & \qquad   +\veps^{2}\dfrac{1}{3!}\tr\left[\partial_{\bd{z}}\bd{Q}_{n}\partial_{\bd{z}}(\nabla^{3}U(\bd{Q}_{n})b^{\veps}\wb{u}_{n}(\bd{p},\bd{Y})Z_{n}^{-1})Z_{n}^{-1}\right]\\
  &\qquad
  -\veps^{2}\dfrac{3}{4!}\tr\left[(\partial_{\bd{z}}\bd{Q}_{n})^{2}\nabla^{4}U(\bd{Q}_{n})b^{\veps}\wb{u}_{n}(\bd{p},\bd{Y})(Z_{n}^{-1})^{2}\right]+R(\bd{x},\bd{q},\bd{p})b^{\veps}\wb{u}_{n,\bd{p}}(\bd{Y}).
\end{split}
\end{equation}
Let us define three operators $L_{0}^{n}$, $L_{1}^{n}$, and $L_{2}^{n}$ acting on $\Phi = \Phi(t, \bd{x}, \bd{y}, \bd{q}, \bd{p})$ by
\begin{align}    L_{0}^{n}(\Phi):=&\Bigl(-\dfrac{1}{2}(-\I\nabla_{\bd{X}}+\bd{P}_{n})^{2}-V(\bd{X})-\partial_{t}S_{n}\Bigr)\Phi\\
  &\notag \qquad +\bd{P}_{n}\cdot\partial_{t}\bd{Q}_{n}\Phi-U(\bd{Q}_{n})\Phi\\
  \notag =&\left(-H_{\bd{P}_{n}}+E_{n}(\bd{P}_{n})\right)\Phi,\\
  L_{1}^{n}(\Phi):=&\Bigl(\I\partial_{t}\Phi-\dfrac{1}{2}\Phi d\Bigr)+\partial_{\bd{z}}\left(\nabla U(\bd{Q}_{n})\Phi Z_{n}^{-1}\right)\\
  &\qquad \notag +\partial_{\bd{z}}\left(\left[(\nabla_{\bd{X}}\Phi+\I\Phi\bd{P}_{n})+(\partial_{t}\bd{P}_{n}-\I\partial_{t}\bd{Q}_{n})\Phi\right]Z_{n}^{-1}\right)\\
  &\qquad \notag +\dfrac{1}{2!}\tr\left[\partial_{\bd{z}}\bd{Q}_{n}(I-\nabla^2 U(\bd{Q}_{n}))\Phi Z_{n}^{-1}\right],\\
  \intertext{and}\label{L2phi}
  L_{2}^{n}(\Phi):=&\dfrac{1}{2!}\tr\left[\partial_{\bd{z}}(\partial_{\bd{z}}((I-\nabla^2 U(\bd{Q}_{n}))\Phi Z_{n}^{-1})Z_{n}^{-1})\right]\\
  &\qquad \notag +\dfrac{2}{3!}\tr\left[\partial_{\bd{z}}(\partial_{\bd{z}}\bd{Q}_{n}\nabla^{3}U(\bd{Q}_{n})\Phi(Z_{n}^{-1})^{2})\right]\\
  &\qquad \notag +\dfrac{1}{3!}\tr\left[\partial_{\bd{z}}\bd{Q}_{n}\partial_{\bd{z}}(\nabla^{3}U(\bd{Q}_{n})\Phi Z_{n}^{-1})Z_{n}^{-1}\right]\\
  &\qquad \notag -\dfrac{3}{4!}\tr\left[(\partial_{\bd{z}}\bd{Q}_{n})^{2}\nabla^{4}U(\bd{Q}_{n})\Phi(Z_{n}^{-1})^{2}\right].
\end{align}
We thus arrive at
\begin{multline}\label{eq:fgaexpansion}
  \bigl(\I\veps\partial_{t}+\tfrac{1}{2}\veps^{2}\Delta-V(\bd{X})
  -U(\bd{x})\bigr)\psi_{\FGA, \infty}^{\veps}\\
  =\dfrac{1}{(2\pi \veps)^{3d/2}} \iint_{\Omega}\int_{\mathbb{R}^{d}}
  \left\{ L_{0}^{n}(b^{\veps}\wb{u}_{n}(\bd{p},\bd{Y}))+\veps L_{1}^{n}(b^{\veps}\wb{u}_{n}(\bd{p},\bd{Y}))\right. \\
  \left.+\veps^{2}L_{2}^{n}(b^{\veps}\wb{u}_{n}(\bd{p},\bd{Y}))
    +R(\bd{x},\bd{q},\bd{p})b^{\veps}\wb{u}_{n}(\bd{p},\bd{Y})\right\}
  G^{\veps}_{\bd{Q}_{n},\bd{P}_{n}}e^{\I S_{n}/\veps}\wb{G}_{\bd{q},\bd{p}}^{\veps}(\bd{y})\psi_{0}(\bd{y})
  \ud\bd{y}\ud\bd{q}\ud\bd{p}.
\end{multline}

Note that by the choice $b_{n,0} = a_{n, 0} u_{n}(\bd{P}_{n}, \bd{X})$, the $\Or(1)$ term in the integrand on the right hand side of \eqref{eq:fgaexpansion} vanishes as
\begin{equation}\label{eq:vanishingO(1)}
  L_0^n ( a_{n, 0}(t, \bd{q}, \bd{p}) u_{n}(\bd{p},\bd{X}) )
  = a_{n, 0}(t, \bd{q}, \bd{p}) \bigl(- H_{\bd{P}_n} + E_n(\bd{P}_{n}) \bigr) u_{n}(\bd{P}_{n}, \bd{X}) = 0
\end{equation}
for any $a_{n, 0}$.

\subsubsection{Leading order term $b_{n, 0}$}
To determine $a_{n,0}$, we set the order $\Or(\veps)$ term on the right
hand side of \eqref{eq:fgaexpansion} to zero and get
\begin{equation}\label{eq:vanishingO(2)}
L_{0}^{n}(b_{n,1}\wb{u}_{n}(\bd{p},\bd{Y}))=-L_{1}^{n}(b_{n,0}\wb{u}_{n}(\bd{p},\bd{Y})).
\end{equation}
We multiply the equation by $\wb{u}_{n}(\bd{P}_{n},\bd{X})$
and integrate over $\Gamma$; this gives
\begin{equation}\label{eq:a0}
  \partial_{t}a_{n,0} = \dfrac{1}{2}a_{n,0}
  \tr\left(\partial_{\bd{z}}\bd{P}_{n}(\nabla_{\bd{P}_{n}}^{2}E_{n})Z_{n}^{-1}\right)
  -\I a_{n,0}\mathcal{A}(\bd{P}_{n})\cdot\nabla_{\bd{Q}_{n}}U
  -\dfrac{\I}{2}a_{n,0}\tr(\partial_{\bd{z}}
  \bd{Q}_{n}(\nabla_{\bd{Q}_{n}}^{2}U)Z_{n}^{-1}).
\end{equation}
Indeed, by integration, we get (index $n$ is suppressed)
\begin{equation}
  \begin{split}
    \int_{\Gamma} & \wb{u}_{n}(\bd{P}_{n},\bd{X})
    \Bigl(-\tfrac{1}{2}(-\I\nabla_{\bd{X}}+\bd{P}_{n})^{2}-V(\bd{X})-\partial_{t}S_{n}\Bigr)
    b_{1}\wb{u}_{n}(\bd{p},\bd{Y}) \ud\bd{X} \\
    & + \int_{\Gamma}\Bigl\{\wb{u}_{n}(\bd{P}_{n},\bd{X})
    \bigl(\I\partial_{t}b_{0}-\dfrac{1}{2}b_{0} d \bigr)
    \wb{u}_{n}(\bd{p},\bd{Y}) +\wb{u}_{n}(\bd{P}_{n},\bd{X})
    \partial_{\bd{z}}\left(\nabla
      U(\bd{Q}_{n}) b_{0}\wb{u}_{n}(\bd{p},\bd{Y})Z_{n}^{-1}\right)\\
    &\qquad \qquad +\wb{u}_{n}(\bd{P}_{n},\bd{X})\partial_{\bd{z}}
    \Bigl(\left[(\nabla_{\bd{X}}b_{0}+\I b_{0}\bd{P}_{n})
      +(\partial_{t}\bd{P}_{n}-\I\partial_{t}\bd{Q}_{n})b_{0}\right]
    \wb{u}_{n}(\bd{p},\bd{Y})Z_{n}^{-1}\Bigr)\\
    & \qquad \qquad +\wb{u}_{n}(\bd{P}_{n},\bd{X})
    \dfrac{1}{2!}\tr\left[\partial_{\bd{z}}\bd{Q}_{n}(I-\nabla^2
      U(\bd{Q}_{n}))b_{0}
      \wb{u}_{n}(\bd{p},\bd{Y})Z_{n}^{-1}\right]\Bigr\}\ud\bd{X}=0.
  \end{split}
\end{equation}
The perpendicular terms in the $b_{j}$'s will now drop out and we can
symplify this equation to
\begin{multline}
  -\braket{u_{n}(\bd{P}_{n},\bd{X}),\partial_{\bd{z}}\left([\I u_{n}(\bd{P}_{n},\bd{X})\nabla_{\mathbf{P}_{n}} E_{n}-\nabla_{\bd{X}}u_{n}(\bd{P}_{n},\bd{X})-\I u_{n}(\bd{P}_{n},\bd{X})\bd{P}_{n}]a_{0}\wb{u}_{n}(\bd{p},\bd{Y}) Z_{n}^{-1}\right)} \\
  +\Bigl( \I\partial_{t}a_{0}
  -a_{0}\mathcal{A}(\bd{P}_{n})\cdot\nabla_{\bd{Q}_{n}}U-\dfrac{d}{2}a_{0}\Bigr)
  \wb{u}_{n}(\bd{p},\bd{Y})
  +\dfrac{1}{2}a_{0}\tr(\partial_{\bd{z}}\bd{Q}_{n}(I-\nabla_{\bd{Q}_{n}}^{2}U)Z_{n}^{-1})
  \wb{u}_{n}(\bd{p},\bd{Y}) =0.
\end{multline}
Using \eqref{eq:diffE}, we observe that
\begin{equation}\label{eq:iden}
\braket{u_{n}(\bd{P}_{n},\bd{X}),[\I u_{n}(\bd{P}_{n},\bd{X})\nabla_{\bd{P}_{n}}E_{n}-\nabla_{\bd{X}}u_{n}(\bd{P}_{n},\bd{X})-\I u_{n}(\bd{P}_{n},\bd{X})\bd{P}_{n}]\cdot\partial_{z}(a_0 \wb{u}_{n}(\bd{P}_{n},\bd{Y})Z_{n}^{-1})}=0.
\end{equation}
Hence, we arrive at
\begin{multline}\label{eq:furthermore}
  a_{0}\tr\bigl(\braket{u_{n}(\bd{P}_{n},\bd{X}),\partial_{\bd{z}} \cdot
    [\I u_{n}(\bd{P}_{n},\bd{X})\nabla_{\bd{P}_{n}} E_{n}-\nabla_{\bd{X}}u_{n}(\bd{P}_{n},\bd{X})
    -\I u_{n}(\bd{P}_{n},\bd{X})\bd{P}_{n}]}Z_{n}^{-1}\bigr) + \\
  +\left( \I\partial_{t}a_{0}-a_{0}\mathcal{A}(\bd{P}_{n})
    \cdot\nabla_{\bd{Q}_{n}}U-\dfrac{d}{2}a_{0}\right)
  +\dfrac{1}{2}a_{0}\tr\bigl(\partial_{\bd{z}}
  \bd{Q}_{n}(I-\nabla_{\bd{Q}_{n}}^{2}U)Z_{n}^{-1}\bigr) =0.
\end{multline}
To further simplify the equation, observe that
\begin{equation}\label{eq:iden2}
  \begin{split}
    & \hspace{-3em} \Bigl\langle u_{n}(\bd{P}_{n},\bd{X}), \partial_{\bd{z}} \cdot [\I u_{n}(\bd{P}_{n},\bd{X})\nabla_{\bd{P}_{n}}E_{n}-\nabla_{\bd{X}}u_{n}(\bd{P}_{n},\bd{X})-\I u_{n}(\bd{P}_{n},\bd{X})\bd{P}_{n}]\Bigr\rangle \\
    &=\I\braket{u_{n}(\bd{P}_{n},\bd{X}),\partial_{\bd{z}}u_{n}(\bd{P}_{n},\bd{X})}(\nabla_{\bd{P}_{n}}E_{n}-\bd{P}_{n})-\braket{u_{n}(\bd{P}_{n},\bd{X}),\partial_{\bd{z}}\nabla_{\bd{X}}u_{n}(\bd{P}_{n},\bd{X})}\\
&\qquad +\I(\partial_{\bd{z}}\nabla_{\bd{P}_{n}}E_{n}-\partial_{\bd{z}}\bd{P}_{n})\\
    &=\I\partial_{\bd{z}}\bd{P}_{n}\braket{u_{n}(\bd{P}_{n},\bd{X}),\partial_{\bd{P}_{n}}u_{n}(\bd{P}_{n},\bd{X})}(\nabla_{\bd{P}_{n}}E_{n}-\bd{P}_{n}) \\ &\qquad -\partial_{\bd{z}}\bd{P}_{n}\braket{u_{n}(\bd{P}_{n},\bd{X}),\nabla_{\bd{p}}\nabla_{\bd{X}}u_{n}(\bd{P}_{n},\bd{X})}_{\Gamma}
 +\I\partial_{\bd{z}}\bd{P}_{n}(\nabla_{\bd{P}_{n}}^{2}E_{n}-I)\\
&\stackrel{\eqref{eq:diffE2}}{=}\dfrac{1}{2}\I\partial_{\bd{z}}\bd{P}_{n}(\nabla_{\bd{P}_{n}}^{2}E_{n}-I).
\end{split}
\end{equation}
Putting this into \eqref{eq:furthermore}, we have
\begin{multline}
  \dfrac{1}{2}\I a_{0}\tr\left(\partial_{\bd{z}}\bd{P}_{n}
    (I-\nabla_{\bd{P}_{n}}^{2}E_{n})Z^{-1}\right)+\left(
    \I\partial_{t}a_{0}-a_{0}\mathcal{A}(\bd{P}_{n})\cdot\nabla_{\bd{Q}_{n}}U
    -\dfrac{d}{2}a_{0}\right)\\
  +\dfrac{1}{2}a_{0}\tr(\partial_{\bd{z}}\bd{Q}_{n}(I-\nabla_{\bd{Q}_{n}}^{2}U)Z_{n}^{-1})]=0.
\end{multline}
We arrive at \eqref{eq:a0} finally by noting that
\begin{equation}\label{eq:iden3}
  \dfrac{1}{2}a\tr\left[\partial_{\bd{z}}\bd{Q}_{n}Z_{n}^{-1}\right]
  +\dfrac{\I}{2}a\tr\left[\partial_{\bd{z}}\bd{P}_{n}Z_{n}^{-1}\right]
  =\dfrac{1}{2}a\tr\left[Z_{n}Z_{n}^{-1}\right]=\dfrac{d}{2}a.
\end{equation}

\subsubsection{Next order term $b_{n, 1}$}

To characterize $b_{n,1}$, we set the order $\Or(\veps^2)$ term in
\eqref{eq:fgaexpansion} to zero, we have
\begin{equation}\label{eq:solvcond}
\displaystyle\int_{\Gamma}\wb{u}_{n}(\bd{P}_{n},\bd{X})\left(L_{0}^{n}(b_{n,2}\wb{u}_{n}(\bd{p},\bd{Y}))+L_{1}^{n}(b_{n,1}\wb{u}_{n}(\bd{p},\bd{Y}))+L_{2}^{n}(b_{n,0}\wb{u}_{n}(\bd{p},\bd{Y}))\right)d\bd{X}=0.
\end{equation}
Let us first derive the equation for $a_{1}$. We start with
\eqref{eq:solvcond} written in expanded form
\begin{equation}\label{eq:exsolvcond}
\begin{split}
  \displaystyle\int_{\Gamma} \wb{u}_{n}(\bd{P}_{n},\bd{X})\biggl\{ & \dfrac{1}{2!} \tr\left[\partial_{\bd{z}}(\partial_{\bd{z}}((I-\nabla^2 U(\bd{Q}_{n}))b_{0}\wb{u}_{n}(\bd{p},\bd{Y})Z_{n}^{-1})Z_{n}^{-1})\right]\\
  & +\dfrac{2}{3!}\tr\left[\partial_{\bd{z}}(\partial_{\bd{z}}\bd{Q}_{n}\nabla^{3}U(\bd{Q}_{n})b_{0}\wb{u}_{n}(\bd{p},\bd{Y})(Z_{n}^{-1})^{2})\right]\\
  &+\dfrac{1}{3!}\tr\left[\partial_{\bd{z}}\bd{Q}_{n}\partial_{\bd{z}}(\nabla^{3}U(\bd{Q}_{n})b_{0}\wb{u}_{n}(\bd{p},\bd{Y})Z_{n}^{-1})Z_{n}^{-1}\right]-\dfrac{3}{4!}\tr\left[(\partial_{\bd{z}}\bd{Q}_{n})^{2}\nabla^{4}U(\bd{Q}_{n})b_{0}\wb{u}_{n}(\bd{p},\bd{Y})(Z_{n}^{-1})^{2}\right]\\
  &+\bigl(\I\partial_{t}b_{1}-\dfrac{1}{2}b_{1}d\bigr)\wb{u}_{n}(\bd{p},\bd{Y})+\partial_{\bd{z}}\left(\nabla U(\bd{Q}_{n})b_{1}\wb{u}_{n}(\bd{p},\bd{Y})Z_{n}^{-1}\right)\\
  &+\partial_{\bd{z}}\left(\left[(\nabla_{\bd{X}}b_{1}+\I b_{1}\bd{P}_{n})+(\partial_{t}\bd{P}_{n}-\I\partial_{t}\bd{Q}_{n})b_{1}\right]\wb{u}_{n}(\bd{p},\bd{Y})Z_{n}^{-1}\right)\\
  &+\dfrac{1}{2!}\tr\left[\partial_{\bd{z}}\bd{Q}_{n}(I-\nabla^2
    U(\bd{Q}_{n}))b_{1}\wb{u}_{n}(\bd{p},\bd{Y})Z_{n}^{-1}\right]+\left.\left(-H_{\bd{P}_{n}}+E(\bd{P}_{n})\right)b_{2}\wb{u}_{n}(\bd{p},\bd{Y})\right\}\,d\bd{X}=0.
\end{split}
\end{equation}
Making use of the Hamiltonian flow \eqref{eq:Hamflow} and the identity
\eqref{eq:iden}, we arrive at
\begin{equation}
\begin{split}
  & -\tr \Bigl(\braket{u_{\bd{P}_{n}}, \partial_{\bd{z}}\cdot[u(\I\nabla E_{n}(\bd{P}_{n}))-\nabla_{\bd{X}}u-\I u\bd{P}_{n}](a_{1})}Z_{n}^{-1}\Bigr)\\
  & \hspace{4em} +\Bigl(\I\partial_{t}a_{1}-a_{1}\mathcal{A}(\bd{P}_{n})\cdot \nabla_{\bd{Q}_{n}}U-\dfrac{d}{2}a_{1}\Bigr) +a_{0}\dfrac{1}{2}\tr\Bigl(\partial_{\bd{z}}(\partial_{\bd{z}}[(I-\nabla_{\bd{Q}_{n}}^{2}U)Z_{n}^{-1}]Z_{n}^{-1})\Bigr)\\
  & \hspace{4em} +a_{1}\dfrac{1}{2}\tr\Bigl(\partial_{\bd{z}}\bd{Q}_{n}
  (I-\nabla_{\bd{Q}_{n}}^{2}U)Z_{n}^{-1}\Bigr) + \dfrac{2}{3!}a_{0}\tr\Bigl(\partial_{\bd{z}}(\partial_{\bd{z}}\bd{Q}_{n}\nabla_{\bd{Q}_{n}}^{3}U(Z_{n}^{-1})^{2})\Bigr)\\
  & \hspace{4em} +\dfrac{1}{3!}a_{0}\tr\Bigl(\partial_{\bd{z}}\bd{Q}_{n}
  \partial_{\bd{z}}(\nabla_{\bd{Q}_{n}}^{3}UZ_{n}^{-1})Z_{n}^{-1}\Bigr)
  -\dfrac{3}{4!}a_{0}\tr\Bigl((\partial_{\bd{z}}\bd{Q}_{n})^{2}
  \nabla_{\bd{Q}_{n}}^{4}U(Z_{n}^{-1})^{2}\Bigr)=0.
\end{split}
\end{equation}
Then using \eqref{eq:iden2} and \eqref{eq:iden3}, upon simplification
we obtain the equation for $a_{n, 1}$
\begin{equation}\label{eq:a1}
\begin{split}
\partial_{t}a_{n,1}=&-\I a_{n,1}\mathcal{A}(\bd{P}_{n})\cdot \nabla_{\bd{Q}_{n}}U+\dfrac{1}{2}a_{n,1}\tr\left(\partial_{\bd{z}}\bd{P}_{n}(\nabla_{\bd{P}_{n}}^{2}E_{n})Z_{n}^{-1}\right)-\dfrac{\I}{2}a_{n,1}\tr\left(\partial_{\bd{z}}\bd{Q}_{n}(\nabla_{\bd{Q}_{n}}^{2}U)Z_{n}^{-1}\right)\\
& +\dfrac{\I}{2}a_{n,0}\tr\left(\partial_{\bd{z}}(\partial_{\bd{z}}[(I-\nabla_{\bd{Q}_{n}}^{2}U)Z_{n}^{-1}]Z_{n}^{-1})\right)
+\dfrac{2\I}{3!}a_{n,0}\tr\left(\partial_{\bd{z}}(\partial_{\bd{z}}\bd{Q}_{n}\nabla_{\bd{Q}_{n}}^{3}U(Z_{n}^{-1})^{2})\right)\\
&+\dfrac{\I}{3!}a_{n,0}\tr\left(\partial_{\bd{z}}\bd{Q}_{n}\partial_{\bd{z}}(\nabla_{\bd{Q}_{n}}^{3}UZ_{n}^{-1})Z_{n}^{-1}\right)-\dfrac{3\I}{4!}a_{n,0}\tr\left((\partial_{\bd{z}}\bd{Q}_{n})^{2}\nabla_{\bd{Q}_{n}}^{4}U(Z_{n}^{-1})^{2}\right).\\
\end{split}
\end{equation}

Define the operator $\mathcal{Q}=\Id-\Pi_{n}$ where $\Pi_{n}$ is the
projection operator onto the nth Bloch wave. $b_{n,1}^{\perp}$
satisfies
$\Pi_{n}b_{n,1}^{\perp}=\braket{u_{n}(\bd{P}_{n},\bd{x}),b_{n,1}^{\perp}}=0$,
and is hence determined by applying $\mathcal{Q}$ to
$L_{0}^{n}(b_{n,1}\wb{u}_{n}(\bd{p},\bd{Y}))=-L_{1}^{n}(b_{n,0}\wb{u}_{n}(\bd{p},\bd{Y}))$. We
obtain
\begin{equation}\label{eq:a1perp}
b_{n,1}^{\perp}\wb{u}_{n}(\bd{p},\bd{Y})=-\left(L_{0}^{n}\right)^{-1}\mathcal{Q}\left(L_{1}^{n}(b_{n,0}\wb{u}_{n}(\bd{p},\bd{Y}))\right).
\end{equation}
Note that the inverse of the operator $L_{0}^{n}$ can be defined on its range.

Thus, we have obtained the equations for $a_{n, 0}$ \eqref{eq:a0},
$a_{n, 1}$ \eqref{eq:a1}, and $b_{n,1}^{\perp}$
\eqref{eq:a1perp}. This can be continued to higher orders.  Let us
summarize the estimate of these terms in the following propositions.

\begin{prop}\label{amplitudebounds}
  For each $k\in\mathbb{N}$, the amplitudes $a_{n,0}$ and $a_{n,1}$,
  given by \eqref{eq:a0} and \eqref{eq:a1} satisfy
  \begin{equation}
    \sup_{t\in[0,T]}M_{k}[a_{n,0}]<\infty,
    \qquad \text{and} \qquad \sup_{t\in[0,T]}M_{k}[a_{n,1}]<\infty.
  \end{equation}
\end{prop}

\begin{proof}
By \eqref{kappabounds}, \eqref{Zinvbound} and \eqref{ubounds}, we see that the right hand side of \eqref{eq:a0} and \eqref{eq:a1} are bounded by some constants independent of $\bd{q}$ and $\bd{p}$ times $a_{n,0}$ and $a_{n,1}$, respectively. An application of Gronwall's inequality yields the result.
\end{proof}

\begin{prop}\label{perpbounds}
  For each $k\in\mathbb{N}$ we have that
  \begin{equation}
    \sup_{t\in[0,T]} M_{k}[b_{n,1}^{\perp} \wb{u}_{n}(\bd{P}_{n},\bd{Y})]<\infty.
  \end{equation}
\end{prop}

\begin{proof}
 The equation for $b_{n,1}^{\perp}$ is given by equation
\eqref{eq:a1perp}. We thus obtain a bound by using the spectrum of $L_{0}^{n}$. We can write
\begin{equation}\label{eq:L0inv}
  (L_{0}^{n})^{-1}(\Phi)=\displaystyle\sum_{m\neq n}\dfrac{\braket{u_{m}(\bd{P}_{m},\cdot),
      \Phi(\cdot,\bd{Y},\bd{q},\bd{p})}_{L^{2}(\Gamma)}u_{m}(\bd{P}_{m},\bd{X})}{E_{n}(\bd{P}_{n})-E_{m}(\bd{P}_{m})}.
\end{equation}
Let $g=\displaystyle\min_{\bd{\xi}\in[-\pi,\pi]^{d}}\{|E_{n}(\bd{\xi})-E_{n-1}(\bd{\xi})|, |E_{n}(\bd{\xi})-E_{n+1}(\bd{\xi})|\}$. Then for each $k\in\mathbb{N}$, we obtain
\begin{equation}
\begin{split}
  M_{k}[b_{n,1}^{\perp}\wb{u}_{n}(\bd{P}_{n},\bd{Y})]&\leq
  M_{k}\left[\dfrac{1}{g} \sum_{m\neq
      n}\braket{u_{m}(\bd{P}_{m},\cdot),
      b_{n,0}(t,\cdot,\bd{q},\bd{p})\wb{u}_{n}(\bd{p},\bd{Y})}_{L^{2}(\Gamma)}u_{n}(\bd{P}_{n},\bd{X})\right]\\
  &=M_{k}\left[\dfrac{\wb{u}_{n}(\bd{p},\bd{Y})}{g}\displaystyle\sum_{m\neq
      n}a_{n,0}(t,\bd{q},\bd{p})\braket{u_{m}(\bd{P}_{m},\cdot),
      u_{n}(\bd{P}_{n},\cdot)}_{L^{2}(\Gamma)}u_{n}(\bd{P}_{n},\bd{X})\right].
\end{split}
\end{equation}
Hence, by Propositions~\ref{ubounds} and \ref{amplitudebounds}, it
suffices to control
\begin{equation}
  M_{k}\biggl[\displaystyle\sum_{m\neq
    n}\braket{u_{m}(\bd{P}_{m},\bd{X}),u_{n}(\bd{P}_{n},\bd{X})}_{L^{2}(\Gamma)}\biggr].
\end{equation}
Since $\int_{\Gamma}|u_{n}(\bd{\xi},\bd{x})|^{2}d\bd{x}=1$, Bessel's
inequality implies that the above is finite.
\end{proof}

\subsection{Proof of Theorem~\ref{theoremA}}

We will need the following estimate, which is proved in~\cite{Ge:02}*{Lemma 2.8}.
\begin{lemma}\label{Remainderlemma}
Suppose $H(\veps)$ is a family of self-adjoint operators for $\veps>0$. Suppose $\psi(t,\veps)$ belongs to the domain of $H(\veps)$, is continuously differentiable in $t$ and approximately solves the Schrodinger equation in the sense that
\begin{equation}
\I\veps\dfrac{\partial\psi}{\partial t}(t,\veps)=H(\veps)\psi(t,\veps)+\zeta(t,\veps),
\end{equation}
where $\zeta(t,\veps)$ satisfies
\begin{equation}
||\zeta(t,\veps)||\leq\mu(t,\veps).
\end{equation}
Then,
\begin{equation}
e^{-\I tH(\veps)/\veps}\psi(0,\veps)-\psi(t,\veps)\leq\veps^{-1}\int_{0}^{t}\mu(s,\veps)ds.
\end{equation}
\end{lemma}

Moreover, for the Fourier integral operator, we have
\begin{lemma}\label{FIObound}
If, for fixed $\bd{x},\bd{y}\in\mathbb{R}^{d}$, $u(\bd{x},\bd{y},\bd{q},\bd{p})\in L^{\infty}(\Omega;\mathbb{C})$, for each $n\in\mathbb{N}$ and any $t$, $\mathcal{I}^{\veps}(u)$ can be extended to a linear bounded operator on $L^{2}(\mathbb{R}^{d},\mathbb{C})$, and we have
\begin{equation}
||\mathcal{I}(u)||_{\mathscr{L}(L^{2}(\mathbb{R}^{d};\mathbb{C}))}\leq||u||_{L^{\infty}(\mathbb{R}^{2d};\mathbb{C})}.
\end{equation}
\end{lemma}

\begin{proof}
The proof of lemma \ref{FIObound} is essentially the same as Proposition 3.7 in ~\cite{LuYa:11} and thus is omitted here.
\end{proof}

We are now ready to prove Theorem~\ref{theoremA}.

\begin{proof}[Proof of Theorem~\ref{theoremA}]
Computing $\I\veps\dfrac{\partial}{\partial t}+\dfrac{1}{2}\veps^{2}\nabla^{2}-V(\bd{X})-U(\bd{x})$ applied to $\mathcal{I}^{\veps}\left(b_{n}^{\veps,1}(t,\dfrac{\bd{x}}{\veps},\bd{q},\bd{p})\wb{u}_{n}(\bd{p},\dfrac{\bd{y}}{\veps})\right)$, we obtain
\begin{equation}
\left(\I\veps\frac{\ud}{\ud t}+\dfrac{1}{2}\veps^{2}\nabla^{2}-V(\bd{X})-U(\bd{x})\right)\mathcal{I}^{\veps}\left(b_{n}^{\veps,1}\wb{u}_{n}(\bd{p},\bd{Y})\right)
=\mathcal{I}^{\veps}\left(\displaystyle\sum_{j=0}^{1}\veps^{j}v_{n,j}\right)+\veps^{2}\mathcal{I}^{\veps}\left(v_{n,2}^{\veps}\right).
\end{equation}
The expressions for $v_{n,0}$, $v_{n,1}$, and $v_{n,2}$ follows from  \eqref{eq:fgaexpansion} by expanding $b^{\veps}$ and the linearity of $L_{0}^{n}$, $L_{1}^{n}$, and $L_{2}^{n}$. By equations \eqref{eq:vanishingO(1)} and \eqref{eq:vanishingO(2)}, $v_{n,0}$ and $v_{n,1}$ vanish.
The remaining term
\begin{equation}\label{eq:remainding}
v_{n,2}^{\veps}=L_{2}^{n}(b_{n}^{\veps,1}\wb{u}_{n}(\bd{p},\bd{Y}))+R(\bd{x},\bd{q},\bd{p})b_{n}^{\veps,1}\wb{u}_{n}(\bd{p},\bd{Y}).
\end{equation}
satisfies $M_{k}[v_{n,2}^{\veps}]<\infty$ by Propositions \ref{Zinvbound}, \ref{amplitudebounds}, and \ref{perpbounds}. Finally, applying Lemma \ref{FIObound} and Lemma \ref{Remainderlemma} we obtain the inequality in Theorem~\ref{theoremA}.
\end{proof}

\bibliographystyle{amsxport}
\bibliography{fga}
\end{document}